%% file: template.tex
\documentclass{article}

\usepackage{arxiv}
\usepackage[utf8]{inputenc} 
\usepackage[T1]{fontenc}    
\usepackage{hyperref}       
\usepackage{url}            
\usepackage{booktabs}       
\usepackage{amsfonts}       
\usepackage{nicefrac}       
\usepackage{microtype}      
\usepackage{lipsum}
\usepackage{amsmath,amssymb}
\usepackage{amsthm}
\usepackage{graphicx}
\usepackage{epstopdf}
\usepackage{listings}
\usepackage{color}
\usepackage{fancyhdr}
\usepackage[framemethod=tikz]{mdframed}
\usepackage{subfig}
\usepackage[nottoc,notlot,notlof]{tocbibind}
\usepackage{csvsimple}
\usepackage{longtable}
\usepackage{colortbl}
\usepackage{float}
\usepackage[]{algorithm2e}
\usepackage{verbatim}
\usepackage{lscape}
\usepackage{multirow}

\include{tex_parts/math}

\title{Choosing the Best Interpolation Data in Images with Noise}

\author{
  Zakaria BELHACHMI \\ 
  IRIMAS\\
  Université de Haute-Alsace\\
  Mulhouse, France \\
  \texttt{zakaria.belhachmi@uha.fr} \\
   \And
  Thomas JACUMIN \\
  IRIMAS \\
  Université de Haute-Alsace\\
  Mulhouse, France \\
  \texttt{thomas.jacumin@uha.fr} \\
}

\begin{document}
\maketitle

\input{tex_parts/abstract}
\input{tex_parts/introduction}

\input{tex_parts/section01}
\input{tex_parts/section02}
\input{tex_parts/section03}
\input{tex_parts/section05}
\input{tex_parts/section06}

\input{tex_parts/conclusion}

\bibliographystyle{unsrt}  
\bibliography{template}

\appendix

\newpage\input{tex_parts/appendix05}
\newpage\input{tex_parts/appendix01}
\newpage\input{tex_parts/appendix02}

\end{document}

%% file: tex_parts/math.tex
\newcommand{\R}{\mathbb{R}}
\newcommand{\N}{\mathbb{N}}

\DeclareMathOperator{\diam}{diam}

\DeclareMathOperator{\capop}{cap}

\DeclareMathOperator{\cl}{cl}

\def\restriction#1#2{\mathchoice
              {\setbox1\hbox{${\displaystyle #1}_{\scriptstyle #2}$}
              \restrictionaux{#1}{#2}}
              {\setbox1\hbox{${\textstyle #1}_{\scriptstyle #2}$}
              \restrictionaux{#1}{#2}}
              {\setbox1\hbox{${\scriptstyle #1}_{\scriptscriptstyle #2}$}
              \restrictionaux{#1}{#2}}
              {\setbox1\hbox{${\scriptscriptstyle #1}_{\scriptscriptstyle #2}$}
              \restrictionaux{#1}{#2}}}
\def\restrictionaux#1#2{{#1\,\smash{\vrule height .8\ht1 depth .85\dp1}}_{\,#2}}

\mdfdefinestyle{thm}{
	nobreak=true,
	hidealllines=true,
	leftline=true,
	innerleftmargin=10pt,
	innerrightmargin=10pt,
	innertopmargin=0pt,
}

\newtheorem{definition}{Definition}[section]\surroundwithmdframed[style=thm]{definition}
\newtheorem{proposition}{Proposition}[section]\surroundwithmdframed[style=thm]{proposition}
\newtheorem{theorem}{Theorem}[section]\surroundwithmdframed[style=thm]{theorem}
\newtheorem{lemma}{Lemma}[section]\surroundwithmdframed[style=thm]{lemma}
\surroundwithmdframed[style=thm]{property}
\newtheorem{problem}{Problem}[section]\surroundwithmdframed[style=thm]{problem}
\newtheorem*{problem*}{Problem}\surroundwithmdframed[style=thm]{problem*}

\theoremstyle{remark}

\newtheorem*{note}{\textbf{Remark}}

%% file: tex_parts/abstract.tex
\begin{abstract}
We introduce and discuss shape based models for finding the best interpolation data in compression of images with noise. The aim is to reconstruct missing regions by means of minimizing data fitting term in the  
$L^2$-norm between the images and their reconstructed counterparts. We analyse the proposed models in the framework of the $\Gamma$-convergence from two different points of view. First, we
consider a continuous stationary PDE model and get pointwise information on the ``relevance'' of each
pixel by a topological asymptotic method. Second, we introduce a finite dimensional setting into
the continuous model based on fat pixels (balls with positive radius), and we study by $\Gamma$-convergence
the asymptotics when the radius vanishes. We extend the method to time-dependent based reconstruction and discuss several strategies for choosing the interpolation data within masks that might be improved over the iterations. Numerical computations are presented that confirm the usefulness of
our theoretical findings for stationary and non-stationary PDE-based image compression.
\end{abstract}

\keywords{image compression \and shape optimization \and $\Gamma$-convergence \and image interpolation \and inpainting \and PDEs \and gaussian noise \and image denoising \and system-aware compression}

%% file: tex_parts/introduction.tex
\section*{Introduction}

Lossy compression plays a important role in many information systems. Nevertheless, most of these methods do not consider any kind of distortion, called noise, on the source signal to compress. It leads to sub-optimal rate-compression performance. To overcome this issue, Dar and Elad in \cite{sysawarecomp} introduced the concept of ``System-Aware Compression'', Figure \ref{fig:system-aware-compression}, where the compression methods take care of noises involved by system sensors. For our study, we suppose $A$ and $B$ in Figure \ref{fig:system-aware-compression} to be identity matrices. \\

\begin{figure}[!ht]
    \centering
    \includegraphics[width=8cm]{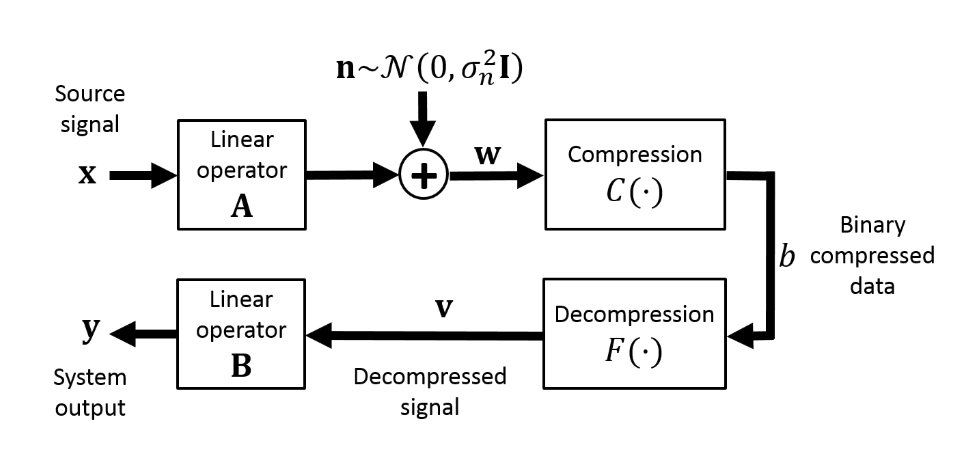}
    \caption{``System-Aware Compression'' description from \cite{sysawarecomp}.}
    \label{fig:system-aware-compression}
\end{figure} 

Using PDEs for image compression is getting more and more interest these past years. However, most of techniques involving PDEs are actually coupled with existing codec such as JPEG \cite{10.5555/574014}, JPEG 2000 \cite{10.5555/2588198} or wavelet transform \cite{BOIX20101265}. Indeed, PDEs are mainly used as pre-filter or post-filter, for example, for image smoothing or denoising \cite{Chan2007, tschumperle:hal-00336540, 999016, 1164922}. 
An interesting idea would be to perform image compression by using a full PDEs-based compression codec \cite{10.1007/11567646_4} by saving a small amount of ``important'' pixels and interpolate the others by suitable PDEs by using image inpainting from a given set of saved parts of the images \cite{10.1145/344779.344972, article1, article2}.
The aim of inpainting is to reconstruct missing parts of the data from known parts, viewed as a Dirichlet boundary condition \cite{ref9, ref10, ref11, 4335633}. \\

Choosing the best interpolation data for image compression, without noise,  have been studied in \cite{ref2} by Belhachmi \textit{et al}, where the aim was to minimize the $H^1$ semi-norm between the original image and its interpolated counterpart. Their work is the only article that proposes a rigorous mathematical analysis to prove the existence of optimal mask and to provide a way to select such  mask. However, because it focuses on edges only, this method gives insufficient results when the input image is affected by noise. Other approaches, mainly stochastic methods have been also proposed \cite{Dell, btree, Schmaltz2014, HoeltgenMHWTSJN15}. Based only on heuristic arguments, these methods do not consider that optimal sets are dependent on the interpolating PDE. In addition, some ideas based on $L^1$ minimization play an important role in recent compressed concepts \cite{1580791}.

In this article, we consider the most important pixels that allows us to minimize the $L^2$-error between the original image and the interpolated one. This point of view leads to pixels choices that reduce the effect of noise, in particular gaussian noise, and as by product, to perform an amount of simultaneous denoising of the considered images. Moreover, we obtain, in the framework of the $\Gamma$-convergence and topological asymptotics, the existence of optimal masks and an effective way to select them. We extend this model to a time-dependent one with two objectifs: first, with a fixed mask, obtained at the initial time, the inpainting yields a denoised image. Second, choosing interpolation data with an adaptive improvement of the mask (which vary in time) allows us to smooth (soften) hard thresholding in the selection and eventually to adjust the size of the mask to a desired accuracy. Numerical experiences, confirm both objectives that is to say an improvement of the inpainting quality and the improvement of the masks during the iterations. \\

In Section \ref{sec:problem_1:continuous_model}, we propose a mathematical model of the problem and its relaxed formulation. In Section \ref{sec:problem_1:topo_grad}, we compute the topological gradient of our minimization problem in order to find a mathematical criterion to construct our set of interpolations points. In Section \ref{sec:problem_1:optimal_distrib}, we change our point of view, by considering ``fat pixels'' instead of a general set of interpolations points. In Section \ref{sec:time-dependent}, we present a second method which is time dependent when the radius vanishes. We extend the method to time-dependent based reconstruction and discuss several strategies for choosing the interpolation data within masks that might be improved over the iterations. Finally, in Section \ref{sec:problem_1:numerical_results}, we expose some numerical results. 

%% file: tex_parts/section01.tex
\section{The Continuous Compression Model}

\label{sec:problem_1:continuous_model}

As said previously, we consider that the input signal is affected by gaussian noise. We begin by giving a short review of the gaussian noise model and the ``maximum-a-posteriori estimate''.

\subsection{Review of the Gaussian Noise Model}

\label{sec:gaussian-noise}

 We consider the probability space $(\Omega,\mathcal{B}(\R),\mathbb{P})$. In this section, we assimilate the image $f$ as a random vector. We consider a random vector $G$ such that each $G_i$ follows a normal distribution with $F_i$ as mean and $\sigma$ standard deviation and all $G_i$ are independent. Then, each $G_i$ satisfies 

\[ \mathbb{P}(G_i\ |\ F_i) = \frac{1}{\sigma\sqrt{2\pi}} e^{-\frac{1}{2\sigma^2}(g_i-f_i)^2}. \]

Since every variable is independent, we have 

\[ \mathbb{P}(G\ |\ F) = \prod_{i=0}^N \mathbb{P}(G_i\ |\ F_i) = \frac{1}{\sigma\sqrt{2\pi}} e^{-\frac{1}{2\sigma^2}\sum_{i=0}^N(g_i-f_i)^2} = \frac{1}{\sigma\sqrt{2\pi}} e^{-\frac{1}{2\sigma^2} \|g-f\|_2^2}. \]

For denoising purpose, we generally use the ``maximum-a-posteriori estimate'' \cite{Wetzstein2017EE3, 6737048} of $f=(f_i)_{i=1}^N$ in $\R^N$ for a given $g$ in $\R^N$, that is

\[ \max_{f\in\R^N} \mathbb{P}(F = f\ |\ G = g), \]

or equivalently 

\[ \max_{f\in\R^N} \ln\mathbb{P}(F\ |\ G). \]

By using Bayes formula, we get  

\begin{align*}
    \max_{f\in\R^N} \ln\mathbb{P}(F\ |\ G) &= \max_{f\in\R^N} \ln\frac{\mathbb{P}(F)\mathbb{P}(G\ |\ F)}{\mathbb{P}(G)} \\
     &= \max_{f\in\R^N} \ln\mathbb{P}(F) + \ln\mathbb{P}(G\ |\ F) - \ln\mathbb{P}(G).
\end{align*}

Since $\mathbb{P}(G)$ does not depend on $f$, we obtain

\[ \max_{f\in\R^N} \mathbb{P}(F\ |\ G) \Leftrightarrow \min_{f\in\R^N} \frac{1}{2\sigma^2} \|g-f\|_2^2 + \ln\mathbb{P}(F). \]

We call $\mathbb{P}(F)$ the prior on images. If we consider a uniform prior on images with normalized intensity values, which is $\mathbb{P}(F_i=f_i) = 1$ if $0\leq f_i \leq 1$ and $0$ otherwise, we have that the maximum-a-posteriori estimate of $F$ is equivalent to 

\[ \min_{f\in [0,1]^N} \|g-f\|_2^2. \]

This result shows us that, minimizing the $L^2$-error is equivalent as removing gaussian noise. As a result, we focus our study on the $L^2$-error in the sequel.

\subsection{The Compression Model}

From this point, we want to minimize the $L^2$-error between the noised input signal and the recovered signal. Actually, we do not minimize the $L^2$-error, but a Tikhonov regularized \cite{Tikhonov1977SolutionsOI} $L^2$-error, that we denote by $J$ in the sequel. We begin by giving a mathematical model of the problem. Let $D$ be a smooth bounded open subset of $\R^2$. We study for, $\alpha>0$, the problem \\

\begin{problem} Find $u$ in $H^1(D)$ such that
	\begin{equation}
		\left\{\begin{array}{rl}
			- \alpha \Delta u + u = 0, & \text{in}\ D\setminus K, \\
			u = f, & \text{in}\ K, \\
			\frac{\partial u}{\partial \mathbf{n}} = 0 & \text{on}\ \partial D. \\
		\end{array}\right .\label{eq:problem_1}
	\end{equation}
	\label{pb:problem_1}
\end{problem}

It is well known that if $u_K \in H^1(D)$ is a solution of Problem \ref{pb:problem_1}, then $u_K$ is a minimizer of 

\begin{equation}
	\min\Big\{ \frac{\alpha}{2}\int_D |\nabla u|^2\ dx + \frac{1}{2} \int_D u^2\ dx\ \Big|\ u\in H^1(D),\ u=f\ \text{a.e. on}\ K \Big\}.
	\label{eq:problem_1:laxmilgram}
\end{equation}

The shape optimization problem we study is

\begin{equation}
	\min_{K\subseteq D,m(K)\leq c}\{ J(u_K)\ |\ u_K\ \text{solution of Problem}\ \ref{pb:problem_1} \},
	\label{pb:problem_1_opt_no_constraint}
\end{equation}

where $J$, defined by
\begin{equation} \label{eq:alpha_error}
	J(u) = \frac{1}{2}\int_D (u-f)^2\ dx + \frac{\alpha}{2}\int_D |\nabla u-\nabla f|^2\ dx,
\end{equation} is called the cost functional, $m$ is a given measure and $c>0$. In fact, the cost functional we choose to study is the $L^2$-norm with a regularization term characterized by $\alpha$.

\subsubsection{Reformulation}

In this section, we want to give a $\max$-$\min$ formulation of the optimization problem \eqref{pb:problem_1_opt_no_constraint}. This new formulation will be more convenient to use later. We have the following proposition : \\

\begin{proposition}
	\label{prop:problem_1_opt_no_constraint:reformulation}
	The optimization problem \eqref{pb:problem_1_opt_no_constraint} is equivalent to 
	\[ \max_{K\subseteq D,m(K)\leq c} \min_{u\in H^1(D), u=f\ \text{in}\ K} \frac{\alpha}{2}\int_D |\nabla u|^2\ dx + \frac{1}{2} \int_D u^2\ dx. \]
\end{proposition}
\begin{proof}
	Let $u_K$ be a solution of Problem \ref{pb:problem_1}. Thus,
	\begin{align*}
		\eqref{pb:problem_1_opt_no_constraint} & \Leftrightarrow \min_{K\subseteq D} \frac{1}{2}\int_D (u_K-f)^2\ dx + \frac{\alpha}{2}\int_D |\nabla u_K-\nabla f|^2\ dx \\
		& \Leftrightarrow \min_{K\subseteq D} \frac{1}{2}\int_D u_K^2\ dx + \frac{\alpha}{2}\int_D |\nabla u_K|^2\ dx - \int_D u_K f\ dx - \alpha\int_D \nabla u_K\cdot\nabla f\ dx.
	\end{align*}

	The weak formulation of Problem \ref{pb:problem_1} gives us $\alpha\int_D \nabla u_K\cdot\nabla f\ dx = \int_D u_K^2\ dx + \alpha\int_D |\nabla u_K|^2\ dx - \int_D u_Kf\ dx$. Hence,

	\begin{align*}
		\eqref{pb:problem_1_opt_no_constraint} & \Leftrightarrow \min_{K\subseteq D} \frac{1}{2}\int_D u_K^2\ dx + \frac{\alpha}{2}\int_D |\nabla u_K|^2\ dx - \int_D u_Kf\ dx - \int_D u_K^2\ dx - \alpha\int_D |\nabla u_K|^2\ dx + \int_D u_Kf\ dx \\
		& \Leftrightarrow \min_{K\subseteq D} -\frac{1}{2}\int_D u_K^2\ dx - \frac{\alpha}{2}\int_D |\nabla u_K|^2\ dx \\
		& \Leftrightarrow \max_{K\subseteq D} \frac{\alpha}{2}\int_D |\nabla u_K|^2\ dx + \frac{1}{2}\int_D u_K^2\ dx.
	\end{align*}

	Using \eqref{eq:problem_1:laxmilgram}, we get the result.
\end{proof}

\begin{note}
	The new formulation of our optimization problem \eqref{pb:problem_1_opt_no_constraint} can be rewritten by penalizing the measure constraint on $K$ as follow :

	\begin{equation}
		\label{pb:problem_1_opt_penalized}
		\max_{K\subseteq D} \min_{u\in H^1(D), u=f\ \text{in}\ K} \frac{\alpha}{2}\int_D |\nabla u|^2\ dx + \frac{1}{2} \int_D u^2\ dx - \beta m(K),
	\end{equation}

	for $\beta >0$.
\end{note}


The well-posedness of \eqref{pb:problem_1_opt_no_constraint} depends of the choice of the measure $m$. In \cite{ref2}, it has been proven that, in the Laplacian case, choosing the $\nu$-capacity as measure $m$ leads to the well-posedness of the optimization problem. Consequently, we will study \eqref{pb:problem_1_opt_no_constraint} when $m$ is the $\nu$-capacity. 

\subsubsection{Framework of the \texorpdfstring{$\gamma$-}-convergence}

For the sake of completeness we recall the definition of the $\nu$-capacity, for $\nu>0$, $\Gamma$-convergence and $\gamma$-convergence written in \cite{ref2}. More details about the $\nu$-capacity or the shape optimization tools can be found in \cite{ref12, ref1}. Let us start with some definitions. \\

\begin{definition}[$\nu$-capacity of a set]
	Let $V\subseteq\R^d$ be a smooth bounded open set and $\nu>0$. We define the $\nu$-capacity of a subset $E$ in $V$ by
	\[ \capop_\nu(E,V) = \inf\Big\{ \int_V |\nabla u|^2\ dx + \nu \int_V u^2\ dx\ \Big|\ u\in H^1_0(V),\ u\geq 1 \text{ a.e. in } E \Big\}. \]
\end{definition}
\begin{note}
	We notice that if, for a given set $E\subseteq V$ and $\nu>0$, we have $\capop_\nu(E,V) = 0$, then we have $\capop_\nu(E,V) = 0$ for every $\nu >0$. Thus, the sets of vanishing capacity are the same for all $\nu >0$. That is why we will drop the $\nu$ and simply write $\capop$ instead of $\capop_\nu$.
\end{note} \vspace{0.2cm}

\begin{definition}[quasi-everywhere property]
	We say that property holds \textbf{quasi-everywhere} if it holds for all $x$ in $E$ except for the elements of a set $Z$ subset of $E$ such that $\capop(Z,E)=0$. We write q.e.
\end{definition}

\begin{definition}[quasi-open set]
	We say that a subset $A$ of $E$ is \textbf{quasi-open} if for every $\varepsilon >0$ there exists an open subset $A_\varepsilon$ of $D$, such that $A\subseteq A_\varepsilon$ and $\capop(A_\varepsilon\setminus A, D)<\varepsilon$.
\end{definition}

We introduce the set $\mathcal{M}_0(D)$ which is denoted by  $\mathcal{M}_0^*(D)$ in \cite{ASNSP_1987_4_14_3_423_0} \\

\begin{definition}[The set $\mathcal{M}_0(D)$]
	We denote by $\mathcal{M}_0(D)$ the set of all non negative Borel measures $\mu$ on $D$, such that

	\begin{itemize}
		\item $\mu(B)=0$, for every Borel set $B$ subset of $D$ with $\capop(B,D) = 0$,
		\item $\mu(B)=\inf\big\{\mu(U)\ \big|\ U\ \text{quasi-open},\ B\subseteq U\big\}$, for every Borel subset $B$ of $D$.
	\end{itemize}
\end{definition}

\begin{definition}[$\nu$-capacity of a measure]
	The $\nu$-capacity, for $\nu>0$, of a measure $\mu$ of $\mathcal{M}_0(D)$ is defined by

	\[ \capop_\nu(\mu, D) := \inf_{u\in H^1_0(D)} \int_D |\nabla u|^2\ dx + \nu\int_D u^2\ dx + \int_D (u-1)^2\ d\mu. \]
\end{definition}

The next proposition will gives us a natural way to identify a set to a measure of $\mathcal{M}_0(D)$. \\

\begin{proposition} Let $E$ be a Borel subset of $D$. We denote by $\infty_E$ the measure of $\mathcal{M}_0(D)$ defined by

	\[ \infty_E(B) := \begin{cases}
		+\infty &,\ \text{if}\ \capop(B,D)>0,\\
		0 &, \ \text{otherwise}.
	\end{cases},\ \text{for all}\ B\ \text{Borel subset of}\ D. \]
\end{proposition}

\begin{note}
	For a given Borel subset $E$ of $D$, we have $\capop_\nu(\infty_E,D) = \capop_\nu(E,D)$.
\end{note}\vspace{0.2cm}

\begin{definition}[$\Gamma$-convergence] 
	Let $V$ be a topological space. We say that the sequence of functionals $(F_n)_n$, from $V$ into $\bar{\R}$, $\Gamma$-converges to $F$ in $V$ if

	\begin{itemize}
		\item for every $u$ in $V$, there exists a sequence $(u_n)_n$ in $V$ such that $u_n\to u$ in $V$ and $F(u)\geq \limsup_{n\to+\infty} F_n(u_n)$,
		\item for every sequence $(u_n)_n$ in $V$ such that $u_n\to u$ in $V$, we have $F(u) \leq \liminf_{n\to +\infty} F_n(u_n)$.
	\end{itemize}	

	We write sometimes $\Gamma-\lim_{n\to+\infty} F_n = F$.
\end{definition}

\begin{definition}[$\gamma$-convergence]
	We say that a sequence $(\mu_n)_n$ of measures in $\mathcal{M}_0(D)$ $\gamma$-converge to a measure $\mu$ in $\mathcal{M}_0(D)$ with respect to $F$ (or $(\mu_n)_n$ $\gamma(F)$-converge to $\mu$) if $F_{\mu_n}$ $\Gamma$-converge in $L^2(D)$ to $F_\mu$.
\end{definition}

We give a locality of the $\gamma$-convergence result, then the $\gamma$-compacity of $\mathcal{M}_0(D)$, from \cite{ASNSP_1997_4_24_2_239_0} and \cite{ASNSP_1987_4_14_3_423_0} respectively. \\

\begin{proposition}[Locality of the $\gamma$-convergence]
	\label{prop:problem_1:gamma_locality}
	Let $(\mu^1_n)_n$ and $(\mu^2_n)_n$ be two sequences of measures in $\mathcal{M}_0(D)$ which $\gamma(F)$-converge to $\mu^1$ and $\mu^2$ respectively. Assume that $\mu^1_n$ and $\mu^2_n$ coincide q.e. on a subset $D'$ of $D$, for every $n\in\N$. Then $\mu^1$ and $\mu^2$ coincide q.e. on $D'$.
\end{proposition}

\begin{proposition}[$\gamma$-compacity of $\mathcal{M}_0(D)$]
	\label{prop:problem_1:M0:gamma_compacity}
	The set $\mathcal{M}_0(D)$ is compact with respect to the $\gamma$-convergence. Moreover, the class of measures of the form $\infty_{D\setminus A}$, with $A$ open and smooth subset of $D$, is dense in $\mathcal{M}_0(D)$.
\end{proposition}

\begin{note}
	The result above means that, for every $\mu$ in $\mathcal{M}_0(D)$, there exists a family of subsets $(E_n)_n$ of $D$, such that $\infty_{E_n}$ $\gamma$-converge to $\mu$.
\end{note}

We will use in the sequel the shape analysis tools that we introduced in this section, in order to study the optimization problem \eqref{pb:problem_1_opt_no_constraint}.

\subsubsection{Analysis}

Now, let us come back to our problem. Thanks to Proposition \ref{prop:problem_1_opt_no_constraint:reformulation}, our optimization problem \eqref{pb:problem_1_opt_no_constraint_pixel} can be rewritten by penalizing the Dirichlet boundary condition $u=f$ in $K$ and penalizing the constraint $\capop(K) \leq c$ as follow

\[ \max_{K\subseteq D} \min_{u\in H^1(D)} \frac{\alpha}{2}\int_D |\nabla u|^2\ dx + \frac{1}{2} \int_D u^2\ dx + \frac{1}{2}\int_D(u-f)^2\ d\infty_K - \beta\capop(\infty_K), \]

where $\beta>0$ depends on $c$. The term $\int_D(u-f)^2\ d\infty_K$ ensure us that $u$ is equals to $f$ in $K$ while the term $- \beta\capop(\infty_K)$ penalize the constraint $\capop(K) \leq c$. Since the optimization problem over $K$ with this kind of problem does not always a solution, we proceed to a relaxation process. We want to study the optimization problem below that we claim is the relaxed problem of the previous problem

\[ \max_{\mu\in\mathcal{M}_0(D)} \min_{u\in H^1(D)} \frac{\alpha}{2}\int_D |\nabla u|^2\ dx + \frac{1}{2} \int_D u^2\ dx + \frac{1}{2}\int_D(u-f)^2\ d\mu - \beta\capop(\mu), \]

where $\mu$ is in $\mathcal{M}_0(D)$. Here we do not want $\mu$ to simply be a characteristic function since a family of characteristic functions does not always converge to a characteristic function, which will be important later. For every $\mu$ in $\mathcal{M}_0(D)$ and $u$ in $H^1(D)$, we define $F_\mu$, from $H^1(D)$ into $\R\cup\{+\infty\}$, by

\[ F_\mu(u) := \begin{cases}
	\alpha\int_D |\nabla u|^2\ dx + \int_D u^2\ dx + \int_D (u-f)^2\ d\mu &,\ \text{if}\ |u| \leq |f|_\infty, \\
	+\infty &,\ \text{otherwise.}
\end{cases} \]

We have that $F_\mu$ is equicoercive with respect to $\mu$, for any $\mu$ in $\mathcal{M}_0(D)$. Indeed, let $u$ be in $H^1(D)$ such that $|u| \leq |f|_\infty$,

\begin{align*}
	F_\mu(u) & = \alpha\int_D |\nabla u|^2\ dx + \int_D u^2\ dx + \int_D (u-f)^2\ d\mu \\
	& = \alpha\int_D |\nabla u|^2\ dx + \int_D u^2\ dx + \int_D u^2\ d\mu + \int_D f^2\ d\mu - 2 \int_D uf\ d\mu \\
	& \geq \alpha\int_D |\nabla u|^2\ dx + \int_D u^2\ dx - 2 \int_D uf\ d\mu \\
	& \geq \alpha\int_D |\nabla u|^2\ dx + \int_D u^2\ dx - 2 |u|_\infty |f|_\infty \mu(D) \\
	& \geq \alpha\int_D |\nabla u|^2\ dx + \int_D u^2\ dx - 2 |f|_\infty^2 \mu(D).
\end{align*}

For every $\mu$ in $\mathcal{M}_0(D)$, we define $E$, from $\mathcal{M}_0(D)$ into $\R$, by

\[ E(\mu) := \min_{u\in H^1(D)} F_\mu(u) = \min_{u\in H^1(D)} \alpha\int_D |\nabla u|^2\ dx + \int_D u^2\ dx + \int_D (u-f)^2\ d\mu. \]

For a given $\mu$ in $\mathcal{M}_0(D)$, $E(\mu)$ correspond to the energy of \\

\begin{problem} Find $u$ in $H^1(D)$ such that
	\begin{equation}
		\left\{\begin{array}{rl}
			- \alpha \Delta u + u + \mu (u-f) = 0, & \text{in}\ D, \\
			\frac{\partial u}{\partial \mathbf{n}} = 0, & \text{on}\ \partial D. \\
		\end{array}\right .\label{eq::dirichlet_penalization:relaxed}
	\end{equation}
	\label{pb:problem_1:dirichlet_penalization:relaxed}
\end{problem}\vspace{0.2cm}

Thus, if $u$ is a solution of Problem \ref{pb:problem_1:dirichlet_penalization:relaxed} for a given $\mu$ in $\mathcal{M}_0(D)$, then $F_\mu(u)= E(\mu)$ and $u$ satisfied the maximum principle $|u| \leq |f|_\infty$. Since we want to include balls centering in $x_0$ in $D$, we do not want $x_0$ to be too close to the boundary of $D$. That is why, we introduce the following notations for $\delta>0$,

\[ D^{-\delta} := \{ x\in D\ |\ d(x,\partial D) \geq \delta \}\subseteq D, \]

\[ \mathcal{K}^\delta(D) := \{ K\subseteq D\ |\ K\ \text{closed},\ K\subseteq D^{-\delta} \}, \]

and

\[ \mathcal{M}_0^\delta(D) := \{ \mu\in\mathcal{M}_0(D)\ |\ \restriction{\mu}{D\setminus D^{-\delta}}=0 \} \subseteq\mathcal{M}_0(D). \]

The problem we study now is the following

\[ \max_{\mu\in\mathcal{M}_0^\delta(D)} \min_{u\in H^1(D)} \frac{\alpha}{2}\int_D |\nabla u|^2\ dx + \frac{1}{2} \int_D u^2\ dx + \frac{1}{2}\int_D(u-f)^2\ d\mu - \beta\capop(\mu). \]

Using the $\gamma$-compacity of $\mathcal{M}_0(D)$, Proposition \ref{prop:problem_1:M0:gamma_compacity} and the locality of the $\gamma(F)$-convergence, Proposition \ref{prop:problem_1:gamma_locality}, we have the result below. \\

\begin{proposition}[$\gamma$-compacity of $\mathcal{M}_0^\delta(D)$]
	The set $\mathcal{M}_0^\delta(D)$ defined above is compact with respect to the $\gamma$-convergence.
	\label{prop:problem_1:M0delta:gamma_compacity}
\end{proposition}

We have the following theorem \\

\begin{theorem}
	We have \[ \cl_\gamma \mathcal{K}_\delta(D) = \mathcal{M}_0^\delta(D), \]
	i.e. $\mathcal{K}_\delta(D)$ is dense into $\mathcal{M}_0^\delta(D)$ with respect to the $\gamma(F)$-convergence.
\end{theorem}
\begin{proof}
	For every $K$ in $\mathcal{K}_\delta(D)$, we have $K\subset D$. Thus $\infty_K$ is in $\mathcal{M}_0^\delta(D)$, so $\mathcal{K}_\delta(D) \subseteq \mathcal{M}_0^\delta(D)$ if we identify a set of $\mathcal{K}_\delta(D)$ by its measure. By the $\gamma$-compacity of $\mathcal{M}_0^\delta(D)$, see Proposition \ref{prop:problem_1:M0delta:gamma_compacity}, we have $\cl_\gamma \mathcal{K}_\delta(D) \subseteq \mathcal{M}_0^\delta(D)$.\\

	Conversely, let $\mu$ be in $\mathcal{M}_0^\delta(D)$. We need to prove that $\mu$ is in $\cl_\gamma \mathcal{K}_\delta(D)$ i.e. $\mu$ is the $\gamma$-limit of elements of $\mathcal{K}_\delta(D)$. This is to be understood as, there exist elements $(K_n)_n$ of $\mathcal{K}_\delta(D)$ such that, $\infty_{K_n}$ $\gamma$-converge to $\mu$. Since $\mathcal{M}_0(D)$ is dense and $\mathcal{M}_0^\delta(D)\subset \mathcal{M}_0(D)$, we know that there exists a sequence $(K_n)_n$ of $\mathcal{P}(D)$ such that, $\infty_{K_n}$ $\gamma$-converge to $\mu$. We need to show that $K_n$ are in $\mathcal{K}_\delta(D)$. By the locality property of the $\gamma$-convergence (Proposition \ref{prop:problem_1:gamma_locality}), we can choose $(K_n)_n$ such that $K_n\subseteq (D^{-\delta})^{1/n}$. Making a homothety $\varepsilon_n K_n$, for $\varepsilon_n<0$ such that $\varepsilon_n K_n \subseteq D^{-\delta}$, we have $\varepsilon_n K_n\in \mathcal{K}_\delta(D)$. Moreover, we can choose $\varepsilon_n$ such that $\varepsilon_n \to 1$, therefore $\varepsilon_n K_n$ $\gamma$-converge to $\mu$.
\end{proof}

Similarly to Lemma 3.4 in \cite{ref2}, we have \\

\begin{theorem}
	Let $\mu_n\in\mathcal{K}_\delta(D)$. If $\mu_n$ $\gamma$-converge to $\mu$, then $\capop_\nu(\mu_n)\to\capop_\nu(\mu)$.
\end{theorem}

\begin{theorem}
	\label{thm:problem_1:convergence_capacity}
	If $(\mu_n)_n$ in $\mathcal{M}_0^\delta(D)$ $\gamma$-converges to $\mu$, then $\mu$ is in $\mathcal{M}_0^\delta(D)$ and $F_{\mu_n}$ $\Gamma$-converges to $F_\mu$ in $L^2(D)$.
\end{theorem}
\begin{proof}
	This proof is similar to the one of Theorem 3.5 in \cite{ref2}. We suppose that $(\mu_n)_n$ in $\mathcal{M}_0^\delta(D)$ $\gamma$-converges to $\mu$. By $\gamma$-compacity, Proposition \ref{prop:problem_1:M0delta:gamma_compacity}, $\mu$ is in $\mathcal{M}_0^\delta(D)$. Now we prove $F_{\mu_n}$ $\Gamma$-converges to $F_\mu$ in $L^2(D)$. \\

	\textbullet ~ $\liminf$ : Let $(u_n)_n$ be a sequence in $H^1(D)$ which converge in $L^2(D)$ to $u$. Let $\varphi\in C^\infty_c(D)$, $0\leq\varphi\leq1$ and $\varphi=1$ in $D^{-\delta}$. Then $(u_n\varphi)_n$ is a sequence in $H^1(D^{-\delta})$ and $u_n\varphi \to_{n\to+\infty} u\varphi$ in $L^2(D)$. Since $(\mu_n)_n$ $\gamma(F)$-converges to $\mu$, we have, in particular for $(\mu_n)_n$,

	\[ \liminf_{n\to +\infty} F_{\mu_n}(u_n\varphi) \geq F_\mu(u\varphi) \]

	i.e.

	\[ \liminf_{n\to+\infty}\Big( \alpha\int_D |\nabla (u_n\varphi)|^2\ dx + \int_D u_n^2\varphi^2\ dx + \int_D (u_n\varphi-f)^2\ d\mu_n \Big) \geq \alpha\int_D |\nabla (u\varphi)|^2\ dx + \int_D u^2\varphi^2\ dx + \int_D (u\varphi-f)^2\ d\mu. \]

	By developping,

	\[ \liminf_{n\to+\infty}\Big( \alpha\int_D |\varphi\nabla u_n|^2\ dx + \alpha\int_D |u_n\nabla\varphi|^2\ dx + 2\alpha\int_D u_n\varphi\nabla u_n\cdot \nabla \varphi\ dx + \int_D u_n^2\varphi^2\ dx + \int_D (u_n\varphi-f)^2\ d\mu_n \Big) \]

	\[ \geq \alpha\int_D |\varphi\nabla u|^2\ dx + \alpha\int_D |u\nabla\varphi|^2\ dx + 2\alpha\int_D u\varphi\nabla u\cdot \nabla \varphi\ dx + \int_D u^2\varphi^2\ dx + \int_D (u\varphi-f)^2\ d\mu. \]

	Eliminating the converging terms, except $\int_D u_n^2\varphi^2\ dx$ and $\int_D u^2\varphi^2\ dx$, we get

	\[ \liminf_{n\to+\infty}\Big( \alpha\int_D |\varphi\nabla u_n|^2\ dx + \int_D u_n^2\varphi^2\ dx + \int_D (u_n\varphi-f)^2\ d\mu_n \Big) \geq \alpha\int_D |\varphi\nabla u|^2\ dx + \int_D u^2\varphi^2\ dx + \int_D (u\varphi-f)^2\ d\mu. \]

	In the left hand side, we use that $0\leq\varphi\leq 1$, and in the right hand side, we take the supremum over all admissible $\varphi$, since the inequality above is true for all $\varphi$,

	\[ \liminf_{n\to+\infty}\Big( \alpha\int_D |\nabla u_n|^2\ dx + \int_D u_n^2\ dx + \int_D (u_n-f)^2\ d\mu_n \Big) \geq \sup_\varphi \Big( \alpha\int_D |\varphi\nabla u|^2\ dx + \int_D u^2\varphi^2\ dx + \int_D (u\varphi-f)^2\ d\mu \Big). \]

	Since $\mu$ is equals to $0$ in $D\setminus D^{-\delta}$, we have 

	\[ \liminf_{n\to+\infty}\Big( \alpha\int_D |\nabla u_n|^2\ dx + \int_D u_n^2\ dx + \int_D (u_n-f)^2\ d\mu_n \Big) \geq \sup_\varphi \Big( \alpha\int_D |\nabla u|^2\varphi^2\ dx + \int_D u^2\varphi^2\ dx \Big) + \int_D (u-f)^2\ d\mu. \]

	We get the $\liminf$ inequality.
	
	\textbullet ~ $\limsup$ : Let $u$ be in $H^1(D)$, such that the principle maximum is fulfilled, i.e. $|u|\leq |f|_\infty$, and $\tilde{u}$ be the extension of $u$ in $H^1_0(D^\delta)$, where $D^\delta$ is the dilatation by a factor $\delta$ of $D$. By the locality property of the $\gamma$-convergence, Property \ref{prop:problem_1:gamma_locality}, we have that $\mu_n$ $\gamma$-converge to $\mu(G)$ in $D^\delta$, where $G$ is

	\[ G_\mu(u) = \int_D |\nabla u|^2\ dx + \int_D u^2\ dx + \varepsilon\int_{D^\delta \setminus D} |\nabla u|^2\ dx + \varepsilon\int_{D^\delta \setminus D} u^2\ dx + \int_D (u-f)^2\ d\mu, \]

	for $\varepsilon >0$. Hence, there exists a sequence $(u_n^\varepsilon)_n$ of $H^1_0(D^\delta)$ such that $u_n^\varepsilon$ converge to $\tilde{u}$ in $L^2(D^\delta)$ and $G_\mu(\tilde{u}) \geq \limsup_{n\to+\infty} G_{\mu_n}(u_n^\varepsilon)$ i.e.

	\[ \int_D |\nabla \tilde{u}|^2\ dx + \int_D \tilde{u}^2\ dx + \varepsilon\int_{D^\delta \setminus D} |\nabla \tilde{u}|^2\ dx + \varepsilon\int_{D^\delta \setminus D} \tilde{u}^2\ dx + \int_D (\tilde{u}-f)^2\ d\mu \]

	\[ \geq \limsup_{n\to+\infty} \int_D |\nabla u_n^\varepsilon|^2\ dx + \int_D (u_n^\varepsilon)^2\ dx + \varepsilon\int_{D^\delta \setminus D} |\nabla u_n^\varepsilon|^2\ dx + \varepsilon\int_{D^\delta \setminus D} (u_n^\varepsilon)^2\ dx + \int_D (u_n^\varepsilon-f)^2\ d\mu_n. \]

	Thus, we have 

	\[ \int_D |\nabla \tilde{u}|^2\ dx + \int_D \tilde{u}^2\ dx + \varepsilon\int_{D^\delta \setminus D} |\nabla \tilde{u}|^2\ dx + \varepsilon\int_{D^\delta \setminus D} \tilde{u}^2\ dx + \int_D (\tilde{u}-f)^2\ d\mu \]

	\[ \geq \limsup_{n\to+\infty} \int_D |\nabla u_n^\varepsilon|^2\ dx + \int_D (u_n^\varepsilon)^2\ dx + \int_D (u_n^\varepsilon-f)^2\ d\mu_n. \]

	Since $\tilde{u}$ is fixed, we make $\varepsilon$ tends to $0$ and extract by a diagonal procedure a subsequence $u_n^{\varepsilon_n}$ converging in $L^2(D^\delta)$ to $\tilde{u}$ i.e.

	\[ \int_D |\nabla \tilde{u}|^2\ dx + \int_D \tilde{u}^2\ dx + \int_D (\tilde{u}-f)^2\ d\mu \geq \limsup_{n\to+\infty} \int_D |\nabla u_n^{\varepsilon_n}|^2\ dx + \int_D (u_n^{\varepsilon_n})^2\ dx + \int_D (u_n^{\varepsilon_n}-f)^2\ d\mu_n. \]

	Setting $u_n := u_n^{\varepsilon_n}|_{D} \in H^1(D)$, we have since $u = \tilde{u}|_D$,

	\[ \int_D |\nabla u|^2\ dx + \int_D u^2\ dx + \int_D (u-f)^2\ d\mu \geq \limsup_{n\to+\infty} \int_D |\nabla u_n|^2\ dx + \int_D (u_n)^2\ dx + \int_D (u_n-f)^2\ d\mu_n. \]
\end{proof}

Finally, here is the main result of this section. \\

\begin{theorem}
	We have \[ \sup_{K\in\mathcal{K}_\delta(D)} \big( E(\infty_K) - \beta\capop_\nu(\infty_K) \big) = \max_{\mu\in\mathcal{M}_0^\delta(D)}\big( E(\mu) - \beta\capop_\nu(\mu) \big). \]
\end{theorem}
\begin{proof}
	Let $(K_n)_n\subset\mathcal{K}_\delta(D)$ be a maximizing sequence of $E(\infty_K) - \beta\capop_\nu(\infty_K)$ i.e.

	\[ \lim_{n\to+\infty} E(\infty_{K_n}) - \beta\capop_\nu(\infty_{K_n}) = \sup_{K\in\mathcal{K}_\delta(D)} \big( E(\infty_K) - \beta\capop_\nu(\infty_K) \big). \]

	One can extract, since $\mathcal{M}_0^\delta(D)$ is a compact metric space when endowed with the distance $d_\gamma$ (Proposition \ref{prop:problem_1:M0delta:gamma_compacity}), from $(\infty_{K_n})_n\subset\mathcal{M}_0^\delta(D)$ a $\gamma$-convergent subsequence. We denote by $\mu_\text{lim}$ this $\gamma$-limit. We know that $\mu_\text{lim}$ is in $\mathcal{M}_0^\delta(D)$ since $\mathcal{M}_0^\delta(D)$ is dense with respect to the $\gamma$-convergence (Proposition \ref{prop:problem_1:M0delta:gamma_compacity}). We denote by $G(\mu_\text{lim})$ the value

	\[ G(\mu_\text{lim}) := \lim_{n\to+\infty} E(\infty_{K_n}) - \beta\capop_\nu(\infty_{K_n}). \]

	By definition of the $\gamma$-convergence, we have $\Gamma-\lim_{n\to+\infty} F_{\infty_{K_n}} = F_{\mu_\text{lim}}$. Since $F_{\infty_{K_n}}$ is equicoercive, we can apply Theorem 7.8 in \cite{ref12}. It leads

	\[ E(\mu_\text{lim}) := \min_{u\in H^1(D)} F_{\mu_\text{lim}}(u) = \lim_{n\to +\infty} \inf_{u\in H^1(D)} F_{\infty_{K_n}}(u) =: \lim_{n\to +\infty} E(\infty_{K_n}). \]

	In addition, by Theorem \ref{thm:problem_1:convergence_capacity} and unicity of the limit, we have 

	\[ \lim_{n\to+\infty} E(\infty_{K_n}) - \beta\capop_\nu(\infty_{K_n}) = G(\mu_\text{lim}) = E(\mu_\text{lim}) - \beta\capop_\nu(\mu_\text{lim}). \]

	Thus, we have



	\begin{align*}
		\sup_{K\in\mathcal{K}_\delta(D)} \big( E(\infty_K) - \beta\capop_\nu(\infty_K) \big) & = \lim_{n\to+\infty} E(\infty_{K_n}) - \beta\capop_\nu(\infty_{K_n}) = E(\mu_\text{lim}) - \beta\capop_\nu(\mu_\text{lim}) \\
		& = \max_{\mu\in\mathcal{M}_0^\delta(D)}\big( E(\mu) - \beta\capop_\nu(\mu) \big).
	\end{align*}
\end{proof}

\begin{note}
	The result above gives us the following information : for every maximizing sequence of sets in $\mathcal{K}_\delta(D)$ of

	\[ \min_{u\in H^1(D)} \alpha\int_D |\nabla u|^2\ dx + \int_D u^2\ dx + \int_D (u-f)^2\ d\infty_K - \beta\capop_\nu(\infty_K), \]

	which corresponds to our original shape optimization problem \eqref{pb:problem_1_opt_no_constraint} where $m$ is the capacity, one can extract a converging subsequence which is the solution of the relaxed problem 

	\[ \max_{\mu\in\mathcal{M}_0^\delta(D)} \min_{u\in H^1(D)} \alpha\int_D |\nabla u|^2\ dx + \int_D u^2\ dx + \int_D (u-f)^2\ d\mu - \beta\capop_\nu(\mu). \]
\end{note}

In the next two sections, we aim to find a way to construct the optimal set $K$.

%% file: tex_parts/section02.tex
\section{Topological Gradient}
\label{sec:problem_1:topo_grad}

Here, we aim to compute the solution of our optimization problem \eqref{pb:problem_1_opt_no_constraint} by using a topological gradient based algorithm as in \cite{Larnier2012, doi:10.1137/S0363012900369538}. This kind of algorithm consists in starting with $K = \bar{D}$ and determine how making small holes in $K$ affect the cost functional, in order to find the balls which have the most decreasing effect. To this end, let us define $K_\varepsilon$ the compact set $K\setminus B(x_0,\varepsilon)$ where $B(x_0,\varepsilon)$ is the ball centered in $x_0\in D$ with radius $\varepsilon>0$ such that $B(x_0,\varepsilon)\subset K$. Let us denote by $j$ the functional \[ j : A\in\mathcal{P}(D) \mapsto \min_{u\in H^1(D), u=f\ \text{in}\ A} \frac{\alpha}{2}\int_D |\nabla u|^2\ dx + \frac{1}{2}\int_D u^2\ dx. \]
Finally, we denote by $u_\varepsilon$ the minimizer of $j(K_\varepsilon)$. Then, we have \\

\begin{proposition} With notations from above, we have when $\varepsilon$ tends to $0$,
	\[ j(K_\varepsilon) - j(K) = \big(f(x_0)-\alpha\Delta f(x_0)\big)^2\pi\varepsilon^4\ln(\varepsilon) + O(\varepsilon^4). \]
	\label{prop:topologicalGradient}
\end{proposition}
\begin{proof}
 	\[ j(K_\varepsilon) - j(K) = \frac{\alpha}{2}\int_{B(x_0,\varepsilon)} |\nabla u_\varepsilon|^2\ dx + \frac{1}{2}\int_{B(x_0,\varepsilon)} u_\varepsilon^2\ dx - \frac{\alpha}{2}\int_{B(x_0,\varepsilon)} |\nabla f|^2\ dx - \frac{1}{2}\int_{B(x_0,\varepsilon)} f^2\ dx. \]

	The weak formulation of Problem \ref{pb:problem_1} leads to 

	\[ j(K_\varepsilon) - j(K) = \frac{\alpha}{2}\int_{B(x_0,\varepsilon)} \nabla u_\varepsilon\cdot\nabla f\ dx + \frac{1}{2}\int_{B(x_0,\varepsilon)} u_\varepsilon f\ dx - \frac{\alpha}{2}\int_{B(x_0,\varepsilon)} |\nabla f|^2\ dx - \frac{1}{2}\int_{B(x_0,\varepsilon)} f^2\ dx. \] 
	Thus
	\begin{align*}
		j(K_\varepsilon) - j(K) & = \frac{\alpha}{2}\int_{B(x_0,\varepsilon)} \nabla (u_\varepsilon-f)\cdot\nabla f\ dx  + \frac{1}{2}\int_{B(x_0,\varepsilon)} (u_\varepsilon-f)\, f\ dx \\
		& = -\frac{\alpha}{2}\int_{B(x_0,\varepsilon)}  (u_\varepsilon-f)\,\Delta f\ dx + \frac{1}{2}\int_{B(x_0,\varepsilon)} (u_\varepsilon-f)\, f\ dx  \\
		& = \frac{1}{2}\int_{B(x_0,\varepsilon)} (u_\varepsilon-f)\,(f-\alpha\Delta f)\ dx 
	\end{align*}
	
    We have $(f-\alpha\Delta f(x)) = (f-\alpha\Delta f)(x_0) + \Vert x - x_0\Vert O(1)$, and hence
    
    \[ j(K_\varepsilon) - j(K)= (f-\alpha\Delta f)(x_0)\int_{B(x_0,\varepsilon)} u_\varepsilon-f\ dx + \varepsilon\,O(1)\int_{B(x_0,\varepsilon)} u_\varepsilon-f\ dx. \]
    
    It is enough to compute the fundamental term in the asymptotic development of the expression $\int_{B(x_0,\varepsilon)} u_\varepsilon-f\ dx$. This is done by using Proposition \ref{prop:appendix:int_calculus} with $w=u_\varepsilon-f$ and $g=-f+\alpha\Delta f$.
\end{proof}

    Since for $\varepsilon < 1$, $\ln\varepsilon < 0$, the result above suggests to keep the points $x_0$ where $\big(f-\alpha\Delta f(x_0)\big)^2$ is maximal, when $\varepsilon$ small enough. From a practical point of view, this is the main result of our local shape analysis. In the next section, we will see that such strict threshold rule might be relaxed. 

	The filter $f-\alpha\Delta f$ is known in image enhancement as a basic filter for image enhancement. It aims to produce more contrasted images. In \cite{ref2}, the use of the $H^1$-semi-norm gives more importance to the laplacian of $f$,and the criterion leads to select only the pixels close to the edges. In our approach, it appears that on one side also pixels ``far'' from the edges may be selected and the neighborhood of the edges is more efficiently restricted with the image enhancement. 

	The mask creation involve the computation of $\Delta f$, which is very sensitive to noise. That is why it is better in practice to smooth the image $f$ before to reduce image noise. 
    %
    %
    %
    %
    %
    %

%% file: tex_parts/section03.tex
\section{Optimal Distribution of Pixels : The ``Fat Pixels'' Approach}
\label{sec:problem_1:optimal_distrib}

In this section, we change our point of view, by considering ``fat pixels'' instead of a general set of interpolation points. In the sequel, we will follow \cite{ref2, ref7}. We restrict our class of admissible sets as an union of balls which represent pixels. For $m>0$ and $n\in\N$, we define

\[ \mathcal{A}_{m,n} := \Big\{ \overline{D}\cap\bigcup_{i=1}^n\overline{B(x_i,r)}\ \Big|\ x_i\in D_r,\ r=mn^{-1/2} \Big\}, \]

where $D_r$ is the $r$-neighbourhood of $D$. The following analysis remains unchanged in $\mathbb{R}^d$, but for the sake of simplicity we restrict ourself to the case $d=2$. We consider problem \eqref{pb:problem_1_opt_no_constraint} for every $K\in\mathcal{A}_{m,n}$ i.e. 

\begin{equation}
	\min_{K\in\mathcal{A}_{m,n}}\Big\{ \frac{1}{2} \int_D (u_K-f)^2\ dx + \frac{\alpha}{2}\int_D |\nabla u_K-\nabla f|^2\ dx\ \Big|\ u_K\ \text{solution of Problem}\ \ref{pb:problem_1} \Big\}.
	\label{pb:problem_1_opt_no_constraint_pixel}
\end{equation}

Here, we do not need to specified a size constraint on our admissible domains. Indeed, imposing $K\in\mathcal{A}_{m,n}$ implies a volume constraint and a geometrical constraint on $K$ since $K$ is formed by a finite number of balls with radius $mn^{-1/2}$. We set $v_K := u_K - f$. Then, $v_K$ is in $H^1(D)$ and satisfies \\

\begin{problem} Find $v$ in $H^1(D)$ such that
	\begin{equation}
		\left\{\begin{array}{rl}
			- \alpha \Delta v + v = g, & \text{in}\ D\setminus K, \\
			v = 0, & \text{in}\ K, \\
			\frac{\partial v}{\partial \mathbf{n}} = 0 & \text{on}\ \partial D, \\
		\end{array}\right .\label{eq:problem_1:v}
	\end{equation}
	\label{pb:problem_1:v}
\end{problem}

where $g := -f+\alpha\Delta f$.

The optimization problem \eqref{pb:problem_1_opt_no_constraint_pixel} can be reformulated as a compliance optimization problem

\begin{equation}
    \label{eq:compliance_pb}
    \min_{K\in\mathcal{A}_{m,n}}\Big\{ \frac{1}{2} \int_D g v_K\ dx\ \Big|\ v_K\ \text{solution of Problem}\ \ref{pb:problem_1:v} \Big\}.
\end{equation}

We deal with Neumann boundary conditions on $D$. However, it is possible to cover the boundary with $\frac{2C_{D}}{m}n^{1/2}$ balls so that we have formally homogeneous Dirichlet boundary conditions on $D$. The well-posedness of such problem have been studied in the laplacian case in \cite{ref7}. Without significant, change we have \\

\begin{theorem}
	If $D$ is an open bounded subset of $\R^2$ and if $g\geq 0$ is in $L^2(D)$, then problem \eqref{eq:compliance_pb} admit a unique solution.
\end{theorem}

If we denote by $K_n^\text{opt}$ the solution, then we have that $\infty_{K_n^\text{opt}}$ $\gamma$-converge to $\infty_D$ as $n$ tends to $+\infty$. However, the numbers of pixels $x_0$ in $D$ to keep goes also to infinity. Thus, it gives no relevant information on the distribution of the points to retain. As pointed out in \cite{ref1}, the local density of $K_n^\text{opt}$ can be obtained by using a different topology for the $\Gamma$-convergence of the rescaled energies. In this new frame, the minimizers are unchanged but their behavior is seen from a different point of view. We define the probability measure $\mu_K$ for a given set $K$ in $\mathcal{A}_{m,n}$ by

\[ \mu_K := \frac{1}{n}\sum_{i=1}^n \delta_{x_i}. \]

We define the functional $F_n$ from $\mathcal{P}(\bar D)$ into $[0,+\infty]$ by

\[ F_n(\mu) := \begin{cases} n\int_D gv_K\ dx &,\ \text{if}\ \exists K\in\mathcal{A}_{m,n},\ \text{s.t.}\ \mu=\mu_K, \\
+\infty&,\ \text{otherwise}. \end{cases} \]

\vspace{0.2cm}

The following $\Gamma$-convergence of $F_n$ theorem is similar to the one given in Theorem 2.2. in \cite{ref7}. \\

\begin{theorem}
    \label{thm:g-convergence}
	If $g\geq 0$, then the sequence of functionals $F_n$, defined above, $\Gamma$-converge with respect to the weak $\star$ topology in $\mathcal{P}(\bar{D})$ to 
	\[ F(\mu) := \int_D \frac{g^2}{\mu_a}\theta(m\mu_a^{1/2})\ dx, \]
	where $\mu = \mu_a dx + \nu$ is the Radon-Nikodym-Lebesgue decomposition of $\mu$ (\cite{ref13}, Theorem 3.8) with respect to the Lebesgue measure and
	\[ \theta(m) := \inf_{K_n\in\mathcal{A}_{m,n}} \liminf_{n\to +\infty} n \int_D gv_{K_n}\ dx, \]
	$v_{K_n}$ solution of Problem \ref{pb:problem_1:v}.
\end{theorem}

As a consequence of the $\Gamma$-convergence stated in the theorem above, the empirical measure $\mu_{K_n^\text{opt}} \to \mu^\text{opt}$ weak $\star$ in $\mathcal{P}(\R^d)$ where $\mu^\text{opt}$ is a minimizer of $F$. Unfortunately, the function $\theta$ is not known explicitly. We establish here after that $\theta$ is positive, non increasing and vanish after some point which will be enough for practical exploration. The next theorem gives an estimate of the function $\theta$ defined above. The proof is given in Appendix \ref{appendix:theta}. \\

\begin{theorem}
	We have, for $m$ in $(0,t_1)$, \[ C_1(\alpha)|\ln(m)| - C_2(\alpha) \leq \theta(m) \leq C_3(\alpha)|\ln(m)|, \]

	where $C_1$, $C_2$ and $C_3$ are constants depending on $\alpha$.
\end{theorem}

\begin{note}
    We can extend the results above to any $g$ since we may formally split the discussion on the sets $\{ g \geq 0 \}$ and $\{ g < 0 \}$.
\end{note} \vspace{0.2cm}

These estimates on $\theta$ suggest that to minimize $F$, when $g^2$ is large, $\mu_a$ should be large in order for $\theta$ to be close to its vanishing point, while when $g^2$ is small $\mu_a$ could be small. Formal Euler-Lagrange equation and the estimates on $\theta$ give the following information : to minimize 

\[ F(\mu) := \int_D \frac{g^2}{\mu_a}\theta(m\mu_a^{1/2})\ dx, \]

one have to take

\[ \frac{\mu_a^2}{|1-\log\mu_a|} \approx c_{m,f} (f-\alpha\Delta f)^2. \]

This introduces a soft-thresholding with respect to the first approach. To sum up, we can choose the interpolation data such that the pixel density is increasing with $g^2 = (f-\alpha\Delta f)^2$. This soft-thresholding rule can be enforced with a standard digital halftoning. According to \cite{ref2},\cite{ref3} and \cite{ref4}, we recall that digital halftoning is a method of rendering that convert a continuous image to a binary image, for example black and white image, while giving the illusion of color continuity. This color continuity is simulated for the human eye by a spacial distribution of black and white pixels. Two different kinds of halftoning algorithms exist : dithering and error diffusion halftoning. The first one is based on a so-called dithering mask function, while the other one is an algorithm which propagate the error between the new value ($0$ or $1$) and the old one (in the interval $[0,1]$) \cite{Floyd1975AnAA}. An ideal digital halftoning method would conserves the average value of gray while giving the illusion of color continuity.

%% file: tex_parts/section05.tex
\section{Time Dependent Compression Method}
\label{sec:time-dependent}

In this section, we propose an other method which depends on time. \\

\begin{problem} For $t$ in $(0,T]$, find $u(t,\cdot)$ in $H^1(D)$ such that
	\begin{equation}
		\left\{\begin{array}{rl}
			\partial_t u(t,\cdot) - \alpha\Delta u(t,\cdot) = 0, & \text{in}\ D\setminus K_t, \\
			u(t,\cdot) = f, & \text{in}\ K_t, \\
			\frac{\partial u}{\partial \mathbf{n}}(t,\cdot) = 0, & \text{on}\ \partial D, \\
		\end{array}\right .\label{eq:problem_1_time}
	\end{equation}
	
	and
	
	\begin{equation*}
		\left\{\begin{array}{rl}
		    & u(0,\cdot) \in L^2(D), \\
		    & K_0 \subset D.
		\end{array}\right .
	\end{equation*}
	
	\label{pb:problem_1_time}
\end{problem}

Using implicit scheme in time, \\

\begin{problem} For $n\in\N$, by knowing $u^n$, find $u^{n+1}$ in $H^1(D)$ such that
	\begin{equation}
		\left\{\begin{array}{rl}
			u^{n+1} - \delta t \alpha \Delta u^{n+1} = u^n, & \text{in}\ D\setminus K_n, \\
			u^{n+1} = f, & \text{in}\ K_n, \\
			\frac{\partial u^{n+1}}{\partial \mathbf{n}} = 0, & \text{on}\ \partial D, \\
		\end{array}\right .\label{eq:problem_1_time_implicit}
	\end{equation}
	
	and
	
	\begin{equation*}
		\left\{\begin{array}{rl}
		    & u^0 \in L^2(D), \\
		    & K_0 \subset D.
		\end{array}\right .
	\end{equation*}

	\label{pb:problem_1_time_implicit}
\end{problem}

The associated energy is

\[ J(u^{n+1}) := \frac{1}{2}\int_D |u^{n+1}|^2\ dx + \frac{\delta t \alpha}{2} \int_D |\nabla u^{n+1}|^2\ dx - \int_D u^n u^{n+1}\ dx. \]

Problem \ref{pb:problem_1_time_implicit} is very close to Problem \ref{pb:problem_1}. The differences are that we have a second member $u^n$ and that the known mask, namely $K_n$, depends on $n$. Since the inpainting mask in our model is time-dependent, we propose, in Section \ref{sec:problem_1:numerical_results}, various methods to construct $K_n$ for $n>0$.

\subsection{Problem Reformulation}

Now, we want to give a $\max$-$\min$ formulation, similar to Proposition \ref{prop:problem_1_opt_no_constraint:reformulation}, of the optimization problem \eqref{pb:problem_1_opt_no_constraint} with Problem \ref{pb:problem_1_time_implicit}. \\

\begin{proposition} Our optimization problem \eqref{pb:problem_1_opt_no_constraint} can be rewritten

    \[ \max_{K^n} \min_{u^{n+1}} \frac{1}{2}\int_D |u^{n+1}|^2 + \frac{\delta t \alpha}{2} \int_D |\nabla u^{n+1}|^2 - \int u^n u^{n+1}. \]
\end{proposition}

\subsection{Topological Gradient}

Like for Section \ref{sec:problem_1:topo_grad}, we use a topological gradient based algorithm to compute the solution of our optimization problem. Again, let us define $K_\varepsilon$ the compact set $K\setminus B(x_0,\varepsilon)$ where $B(x_0,\varepsilon)$ is the ball centered in $x_0\in D$ with radius $\varepsilon>0$ such that $B(x_0,\varepsilon)\subset K$. Let us denote by $j$ the functional, for $u^n$ in $H^1(D)$ given,

\[ j : A\in\mathcal{P}(D) \mapsto \min_{u\in H^1(D), u=f\ \text{in}\ A} \frac{\delta t \alpha}{2}\int_D |\nabla u|^2\ dx + \frac{1}{2}\int_D u^2\ dx - \int_D u^n u\ dx. \]

Finally, we denote by $u_\varepsilon$ the minimizer of $j(K_\varepsilon)$. Then, we have \\

\begin{theorem} With notations from above, we have when $\varepsilon$ tends to $0$,

    \[ j(K_\varepsilon) - j(K) = \big(f(x_0)-\delta t \alpha\Delta f(x_0) - u(x_0)\big)^2\pi\varepsilon^4\ln(\varepsilon) + O(\varepsilon^4). \]
\end{theorem}
\begin{proof}
    For simplicity, we write $u_\varepsilon$ instead of $u_\varepsilon^{n+1}$ and $u$ instead of $u^n$. Thus, we have

    \[ j(K_\varepsilon) - j(K) = \frac{\delta t \alpha}{2}\int_{B(x_0,\varepsilon)} |\nabla u_\varepsilon|^2\ dx + \frac{1}{2}\int_{B(x_0,\varepsilon)} u_\varepsilon^2\ dx - \int_{B(x_0,\varepsilon)} u u_\varepsilon\ dx \] \[\hspace{8cm}- \frac{\delta t \alpha}{2}\int_{B(x_0,\varepsilon)} |\nabla f|^2\ dx - \frac{1}{2}\int_{B(x_0,\varepsilon)} f^2\ dx + \int_{B(x_0,\varepsilon)} uf\ dx. \]
    
    The variational formulation gives us 
    
    \[ \int_{B(x_0,\varepsilon)} u_\varepsilon^2\ dx + \delta t \alpha\int_{B(x_0,\varepsilon)} |\nabla u_\varepsilon|^2\ dx = \int_{B(x_0,\varepsilon)} u_\varepsilon f\ dx + \delta t \alpha \int_{B(x_0,\varepsilon)}\nabla u_\varepsilon\cdot\nabla f\ dx + \int_{B(x_0,\varepsilon)} u(u_\varepsilon-f)\ dx. \]
    
    Then, 
    
    \begin{align*}
        j(K_\varepsilon) - j(K) &= \frac{1}{2}\int_{B(x_0,\varepsilon)} u_\varepsilon f\ dx + \frac{\delta t \alpha}{2} \int_{B(x_0,\varepsilon)}\nabla u_\varepsilon\cdot\nabla f\ dx - \frac{1}{2}\int_{B(x_0,\varepsilon)} u(u_\varepsilon-f)\ dx \\
        & \hspace{8cm} - \frac{\delta t \alpha}{2}\int_{B(x_0,\varepsilon)} |\nabla f|^2\ dx - \frac{1}{2}\int_{B(x_0,\varepsilon)} f^2\ dx \\
        & = \frac{1}{2}\int_{B(x_0,\varepsilon)} f(u_\varepsilon - f)\ dx + \frac{\delta t \alpha}{2}\int_{B(x_0,\varepsilon)} \nabla f\cdot\nabla(u_\varepsilon - f)\ dx - \frac{1}{2}\int_{B(x_0,\varepsilon)} u(u_\varepsilon - f)\ dx \\
        & = \frac{1}{2}\int_{B(x_0,\varepsilon)} f(u_\varepsilon - f)\ dx - \frac{\delta t \alpha}{2}\int_{B(x_0,\varepsilon)} \Delta f(u_\varepsilon - f)\ dx - \frac{1}{2}\int_{B(x_0,\varepsilon)} u(u_\varepsilon - f)\ dx \\
        & = \frac{1}{2}\int_{B(x_0,\varepsilon)} (f-\delta t \alpha\Delta f - u) (u_\varepsilon - f)\ dx.
    \end{align*}

    We have $(f-\delta t \alpha\Delta f - u) = (f-\delta t \alpha\Delta f - u)(x_0) + \Vert x - x_0\Vert O(1)$, and hence
    
    \[ j(K_\varepsilon) - j(K) = \frac{1}{2}\Big(f(x_0)-\delta t \alpha\Delta f(x_0) - u(x_0)\Big) \int_{B(x_0,\varepsilon)} (u_\varepsilon - f)\ dx + O(\varepsilon)\int_{B(x_0,\varepsilon)} (u_\varepsilon - f)\ dx. \]
   
    Ones again, it is enough to compute the fundamental term in the asymptotic development of the expression $\int_{B(x_0,\varepsilon)} u_\varepsilon-f\ dx$. This is done by using Proposition \ref{prop:appendix:int_calculus} with, $w=u_\varepsilon-f$ and $g=-f+\delta t \alpha\Delta f + u$.
\end{proof}

This result conclude the theoretical part of this paper. In the remaining part, we confront our results to real cases.

%% file: tex_parts/section06.tex
\section{Numerical Results}
\label{sec:problem_1:numerical_results}

In this last section, we present some numerical results. We begin by comparing different masks to confirm our work from Section \ref{sec:problem_1:topo_grad} and Section \ref{sec:problem_1:optimal_distrib}. Then, we present four time-dependent methods, that use results from Section \ref{sec:time-dependent}, and compare them. Finally, we propose two different ways to deal with colored images. We use finite differences to compute approximated solutions of involved problems.

\subsection{Masks Comparison}

We compare the $L^2$-error of the interpolation for different masks using Problem \ref{pb:problem_1}. 
For these experiments, we use the well-known grayscale image called ``Lena'', which size is $256\times256$ pixels, Figure \ref{fig:experiments:masks:inpainting:white:0:Lenna} (a). We apply gaussian noise of different deviation, denoted by $\sigma$, and we compare the following masks : ``Optimized'' corresponds to the mask derived from the topological gradient in Section \ref{sec:problem_1:topo_grad}, ``Halftoned Optimized'' corresponds to the mask derived from Section \ref{sec:problem_1:optimal_distrib}, ``H1'' correspond to the mask found in \cite{ref2}, ``Halftoned H1'' correspond to the halftoned mask found in \cite{ref2}, ``Random'' is a mask where pixels are selected randomly and ``B-Tree'' is the mask described in \cite{btree}. We denote 
by $f$ the original image,
by $f_\text{n}$ the noised image,
by $u$ the reconstructed image,
by ``Norm'' the error $\|f_\text{n}-u\|_{L^2(D)}$ and
by $\sharp K$ the number of pixels saved in the inpainting mask $K$.
Figure \ref{fig:experiments:masks:inpainting:white:0:Lenna} and Figure \ref{fig:experiments:masks:inpainting:white:0.1:Lenna} are reconstructions for $\alpha = 3$, with and without noise respectively.

\begin{center}\begin{longtable}{|c|c|c|c|c|c|c|c|c|c|c|c|c|}
	\hline
    \multirow{2}{*}{$\sigma$} & \multicolumn{2}{c|}{\textbf{Optimized}} & \multicolumn{2}{c|}{\textbf{\parbox{1.6cm}{\textbf{Halftoned Optimized}}}} & \multicolumn{2}{c|}{\textbf{H1}} & \multicolumn{2}{c|}{\textbf{\parbox{1.7cm}{\textbf{Halftoned H1}}}} & \multicolumn{2}{c|}{\textbf{Random}} & \multicolumn{2}{c|}{\textbf{B-Tree}} \\
    \cline{2-13}
    & Norm & $\sharp K$ & Norm & $\sharp K$ & Norm & $\sharp K$ & Norm & $\sharp K$ & Norm & $\sharp K$ & Norm & $\sharp K$ \\
	\hline
	\endhead
	\endfoot
	\begin{filecontents*}{resources/tables/data-c-0.1-masks.csv}
1,2,3,4,5,6,7,8,9,10,11,12,13,14,15,16,17
Lenna,0.00,0,0,100.13,6553,51.66,6488,117.80,6553,93.80,6492,74.58,6624,56.91,6597,3.00
Lenna,0.05,0,0,73.24,6553,47.63,6494,76.05,6553,60.26,6508,69.49,6615,55.69,6528,3.00
Lenna,0.10,0,0,67.61,6553,51.08,6492,76.14,6553,63.40,6504,72.78,6560,62.24,6510,3.00
\end{filecontents*}
\csvreader[
		late after line=\csvifoddrow{\\\rowcolor{white}}{\\\rowcolor{gray!25}},
		late after last line=\\,
		before reading={},
		after reading={}
	] {resources/tables/data-c-0.1-masks.csv}
	{1=\Name, 2=\WNDev,3=\Salt, 4=\Pepper, 5=\EOpti, 6=\COpti, 7=\EOptiHalf, 8=\COptiHalf, 9=\EBel, 10=\CBel, 11=\EBelHalf, 12=\CBelHalf, 13=\ERand, 14=\CRand, 15=\EBTree, 16=\CBTree, 17=\Alpha}
	{\WNDev & \textbf{\EOpti} & \textbf{\COpti} & \textbf{\EOptiHalf} & \textbf{\COptiHalf} & \EBel & \CBel & \EBelHalf & \CBelHalf & \ERand & \CRand & \EBTree & \CBTree}
    \hline
	\caption{$L^2$-error comparison for different masks by taking $10\%$ of pixels, and applying Problem \ref{pb:problem_1} with $\alpha=3$.}
	\label{table:l2-masks}
\end{longtable}\end{center}

\begin{figure}[H]
	\centering
	\subfloat[Original image.]{
		\includegraphics[height=3.5cm]{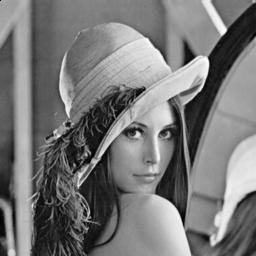}
	}
	\qquad
	\subfloat[Optimized mask.]{
		\includegraphics[height=3.5cm]{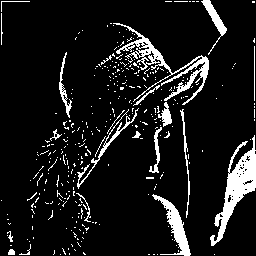}
	}
	\qquad
	\subfloat[Reconstruction with the optimized mask. $\|f-u\|_{L^2(D)}=100.13$, $\sharp K = 6553$.]{
		\includegraphics[height=3.5cm]{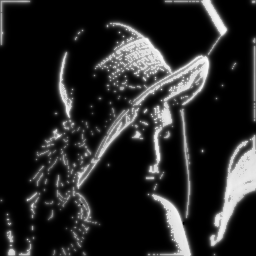}
	}
	\qquad
	\subfloat[Halftoned optimized mask.]{
		\includegraphics[height=3.5cm]{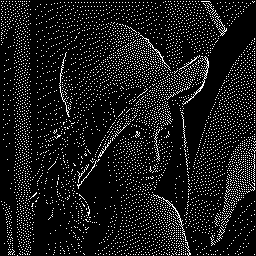}
	}
	\qquad
	\subfloat[Reconstruction with the halftoned optimized mask. $\|f-u\|_{L^2(D)}=51.66$, $\sharp K = 6488$.]{
		\includegraphics[height=3.5cm]{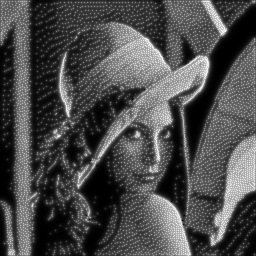}
	}
	\qquad
	\subfloat[Reconstruction with the ``H1'' mask. $\|f-u\|_{L^2(D)}=117.80$, $\sharp K = 6553$.]{
		\includegraphics[height=3.5cm]{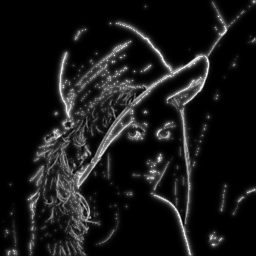}
	}
	\qquad
	\subfloat[Reconstruction with the ``Halftoned H1'' mask. $\|f-u\|_{L^2(D)}=93.80$, $\sharp K = 6492$.]{
		\includegraphics[height=3.5cm]{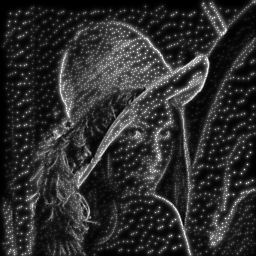}
	}
	\qquad
	\subfloat[Reconstruction with the random mask. $\|f-u\|_{L^2(D)}=74.58$, $\sharp K = 6624$.]{
		\includegraphics[height=3.5cm]{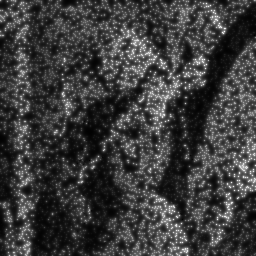}
	}
	\qquad
	\subfloat[Reconstruction with the B-Tree mask. $\|f-u\|_{L^2(D)}=56.91$, $\sharp K = 6597$.]{
		\includegraphics[height=3.5cm]{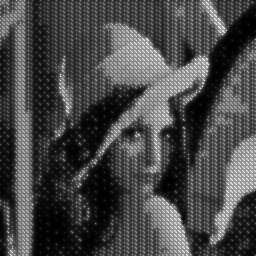}
	}
	
	\caption{Masks comparison for Problem \ref{pb:problem_1} by saving $10\%$ of pixels.}
	\label{fig:experiments:masks:inpainting:white:0:Lenna}
\end{figure} 

\begin{figure}[H]
	\centering
	\subfloat[Original image with a gaussian noise of deviation $0.1$.]{
		\includegraphics[height=3.5cm]{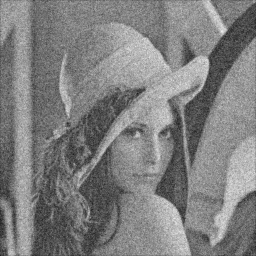}
	}
	\qquad
	\subfloat[Optimized mask.]{
		\includegraphics[height=3.5cm]{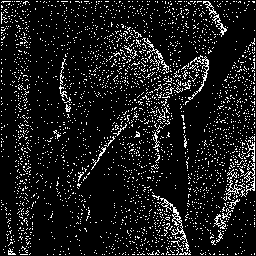}
	}
	\qquad
	\subfloat[Reconstruction with the optimized mask. $\|f_\text{n}-u\|_{L^2(D)}=67.61$, $\sharp K = 6553$.]{
		\includegraphics[height=3.5cm]{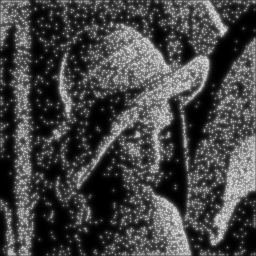}
	}
	\qquad
	\subfloat[Halftoned optimized mask.]{
		\includegraphics[height=3.5cm]{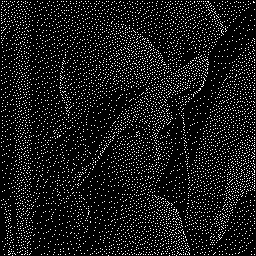}
	}
	\qquad
	\subfloat[Reconstruction with the halftoned optimized mask. $\|f_\text{n}-u\|_{L^2(D)}=51.08$, $\sharp K = 6492$.]{
		\includegraphics[height=3.5cm]{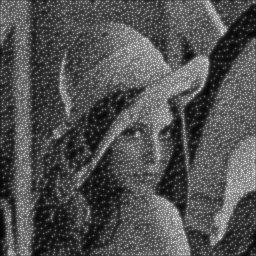}
	}
	\qquad
	\subfloat[Reconstruction with the ``H1'' mask. $\|f_\text{n}-u\|_{L^2(D)}=76.14$, $\sharp K = 6553$.]{
		\includegraphics[height=3.5cm]{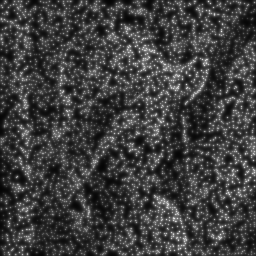}
	}
	\qquad
	\subfloat[Reconstruction with the ``Halftoned H1'' mask. $\|f_\text{n}-u\|_{L^2(D)}=63.40$, $\sharp K = 6504$.]{
		\includegraphics[height=3.5cm]{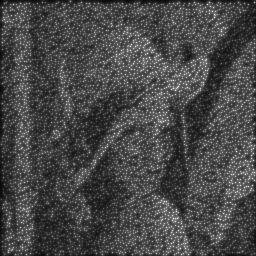}
	}
	\qquad
	\subfloat[Reconstruction with the random mask. $\|f_\text{n}-u\|_{L^2(D)}=72.78$, $\sharp K = 6560$.]{
		\includegraphics[height=3.5cm]{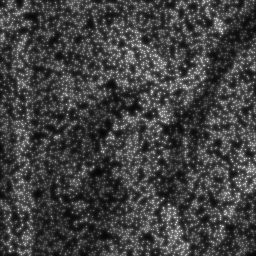}
	}
	\qquad
	\subfloat[Reconstruction with the B-Tree mask. $\|f_\text{n}-u\|_{L^2(D)}=62.24$, $\sharp K = 6510$.]{
		\includegraphics[height=3.5cm]{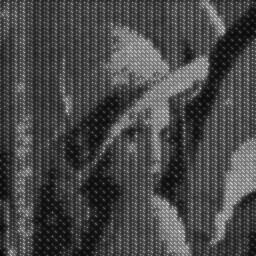}
	}
	
	\caption{Masks comparison with gaussian noise for Problem \ref{pb:problem_1} by saving $10\%$ of pixels.}
	\label{fig:experiments:masks:inpainting:white:0.1:Lenna}
\end{figure}

As expected from our mathematical analysis, masks ``Optimized'' and ``Halftoned Optimized'' globally gives the lowest $L^2$-error for Problem \ref{pb:problem_1} and the best visual results. The two ``H1'' masks are not efficient coupled with Problem \ref{pb:problem_1} since they have been designed for the Homogeneous Diffusion inpainting (i.e. for $\alpha\to +\infty$). Moreover, since these masks only depends on $|\Delta f|^2$, they are very sensitive to gaussian noise. The B-Tree algorithm, by its nature, induce visible artifacts, which leads to this mosaic visual.

\subsection{Methods Comparison}

Now, we propose different methods to construct mask and reconstruct image and compare them to the method described in \cite{ref2}, denoted ``H1'' in the sequel, i.e. the ``H1'' mask coupled with homogeneous diffusion inpainting. The method denoted by ``L2'' is the one with the halftoned mask from Section \ref{sec:problem_1:optimal_distrib} along with Problem \ref{pb:problem_1}. The remaining methods use the parabolic Problem \ref{pb:problem_1_time_implicit} as follow : 

\textbullet\hspace{0.1cm} Encoding : For the ``L2Sta'' method, we use the same mask, for all $n$ in $\{0,\dots,N\}$, as for the ``L2'' method, i.e. the halftoned mask from Section \ref{sec:problem_1:optimal_distrib}. For the ``L2Dec'' method, the ``L2Inc'' method and the ``L2Insta'' method, we use algorithms described in Appendix \ref{appendix:algos}.

\textbullet\hspace{0.1cm} Decoding : For each methods, we use only the last $K_n$ from the encoding step and inpaint with Problem  \ref{pb:problem_1_time_implicit} until a fixed time.


\begin{center}\begin{longtable}{|c|c|c|c|c|c|c|c|c|c|c|c|c|}
	\hline
    \multirow{2}{*}{$\sigma$} & \multicolumn{2}{c|}{\textbf{H1}} & \multicolumn{2}{c|}{\textbf{L2}} & \multicolumn{2}{c|}{\textbf{L2Sta}} & \multicolumn{2}{c|}{\textbf{L2Dec}} & \multicolumn{2}{c|}{\textbf{L2Inc}} & \multicolumn{2}{c|}{\textbf{L2Insta}} \\
    \cline{2-13}
     & Norm & $\sharp K$ & Norm & $\sharp K$ & Norm & $\sharp K$ & Norm & $\sharp K$ & Norm & $\sharp K$ & Norm & $\sharp K$ \\
	\hline
	\endhead
	\endfoot
	\begin{filecontents*}{resources/tables/data-c-0.1-cmp.csv}
1,2,3,4,5,6,7,8,9,10,11,12,13,14
Lenna,0.00,11.95,6492,47.91,6486,24.50,6488,31.61,6553,12.99,6520,9.63,6514
Lenna,0.10,14.04,6503,46.06,6492,16.39,6484,14.61,6553,10.52,6520,13.19,6491
\end{filecontents*}
\csvreader[
		late after line=\csvifoddrow{\\\rowcolor{white}}{\\\rowcolor{gray!25}},
		late after last line=\\,
		before reading={},
		after reading={}
	] {resources/tables/data-c-0.1-cmp.csv}
	{1=\Name, 2=\WNDev, 3=\EH1, 4=\NH1, 5=\E, 6=\N, 7=\EConst, 8=\NConst, 9=\EDec, 10=\NDec, 11=\EInc, 12=\NInc, 13=\ESta, 14=\NSta}
	{\WNDev & \EH1 & \NH1 & \E & \N & \EConst & \NConst & \EDec & \NDec & \EInc & \NInc & \ESta & \NSta}
    \hline
	\caption{$L^2$-error comparison for different masks by taking $10\%$ of pixels for the following methods : ``L2'' with $\alpha=3.61$, ``L2Sta'' with $\alpha=10$, ``L2Dec'' with $\alpha=30$, $N=35$, ``L2Inc'' with $\alpha=8$, $N=40$ and ``L2Insta'' with $\alpha=10$, $N=10$.}
	\label{table:methods:cmp}
\end{longtable}\end{center}

Figure \ref{fig:experiments:method:cmp:wn:0} and Figure \ref{fig:experiments:method:cmp:wn:0.1} are examples of reconstructed image with methods described before, without noise and with a gaussian noise of deviation $\sigma=0.1$ respectively.

\begin{figure}[H]
	\centering
	\subfloat[Original image.]{
		\includegraphics[height=3.5cm]{resources/images/Lenna.png}
	}
	\qquad
	\subfloat[Reconstruction using method ``H1''. $\|u-f\|_{L^2(D)}=11.95$, $\sharp K=6492$.]{
		\includegraphics[height=3.5cm]{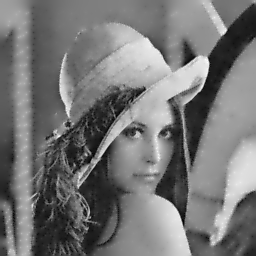}
	}
	\qquad
	\subfloat[Reconstruction using method ``L2''. $\|u-f\|_{L^2(D)}=47.91$, $\sharp K=6486$.]{
		\includegraphics[height=3.5cm]{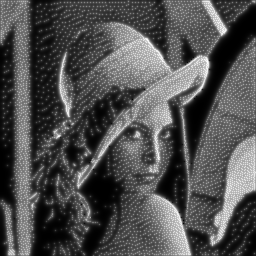}
	}
	\qquad
	\subfloat[Reconstruction using method ``L2Sta''. $\|u-f\|_{L^2(D)}=24.50$, $\sharp K=6488$.]{
		\includegraphics[height=3.5cm]{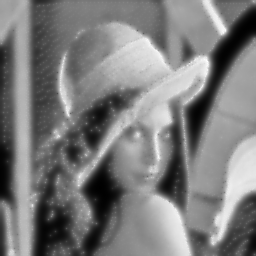}
	}
	\qquad
	\subfloat[Reconstruction using method ``L2Dec''. $\|u-f\|_{L^2(D)}=31.61$, $\sharp K=6553$.]{
		\includegraphics[height=3.5cm]{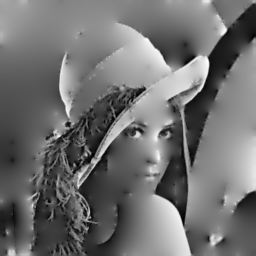}
	}
	\qquad
	\subfloat[Reconstruction using method ``L2Inc''. $\|u-f\|_{L^2(D)}=12.99$, $\sharp K=6520$.]{
		\includegraphics[height=3.5cm]{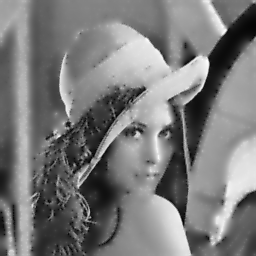}
	}
	\qquad
	\subfloat[Reconstruction using method ``L2Insta''. $\|u-f\|_{L^2(D)}=9.63$, $\sharp K=6514$.]{
		\includegraphics[height=3.5cm]{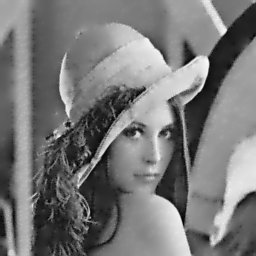}
	}
	
	\caption{Image reconstruction using various methods, with $10\%$ of total pixels saved. ``L2'' with $\alpha=3.61$, ``L2Sta'' with $\alpha=10$, ``L2Dec'' with $\alpha=30$, $N=35$, ``L2Inc'' with $\alpha=8$, $N=40$ and ``L2Insta'' with $\alpha=10$, $N=10$.}
	\label{fig:experiments:method:cmp:wn:0}
\end{figure} 

\begin{figure}[H]
	\centering
	\subfloat[Original image with a gaussian noise of deviation $0.1$.]{
		\includegraphics[height=3.5cm]{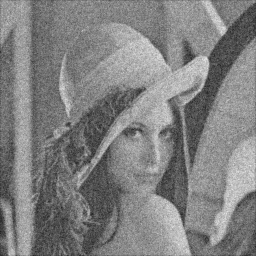}
	}
	\qquad
	\subfloat[Reconstruction using method ``H1''.  $\|u-f_\text{n}\|_{L^2(D)}=14.04$, $\sharp K=6503$.]{
		\includegraphics[height=3.5cm]{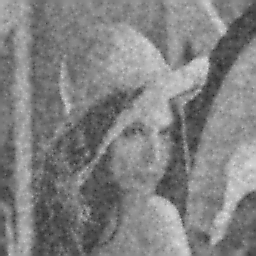}
	}
	\qquad
	\subfloat[Reconstruction using method ``L2''.  $\|u-f_\text{n}\|_{L^2(D)}=46.06$, $\sharp K=6492$.]{
		\includegraphics[height=3.5cm]{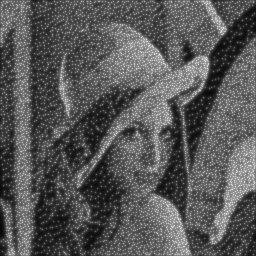}
	}
	\qquad
	\subfloat[Reconstruction using method ``L2Sta''.  $\|u-f_\text{n}\|_{L^2(D)}=16.39$, $\sharp K=6484$.]{
		\includegraphics[height=3.5cm]{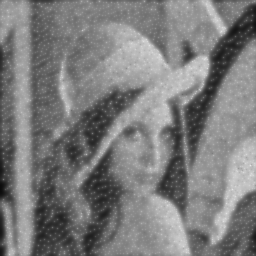}
	}
	\qquad
	\subfloat[Reconstruction using method ``L2Dec''.  $\|u-f_\text{n}\|_{L^2(D)}=14.61$, $\sharp K=6553$.]{
		\includegraphics[height=3.5cm]{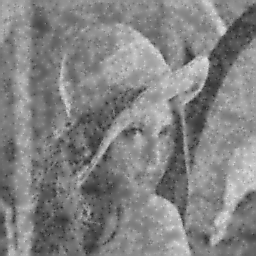}
	}
	\qquad
	\subfloat[Reconstruction using method ``L2Inc''.  $\|u-f_\text{n}\|_{L^2(D)}=10.52$, $\sharp K=6520$.]{
		\includegraphics[height=3.5cm]{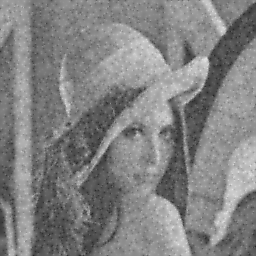}
	}
	\qquad
	\subfloat[Reconstruction using method ``L2Insta''.  $\|u-f_\text{n}\|_{L^2(D)}=13.19$, $\sharp K=6491$.]{
		\includegraphics[height=3.5cm]{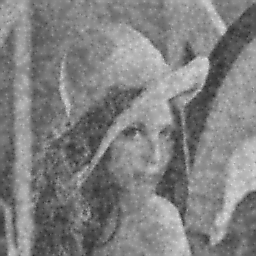}
	}
	
	\caption{Image reconstruction using various methods, with $10\%$ of total pixels saved and a gaussian noise of deviation $\sigma=0.1$ applied to the input image.}
	\label{fig:experiments:method:cmp:wn:0.1}
\end{figure} 

The ``H1'' method seems to give a nice visual quality on Figure \ref{fig:experiments:method:cmp:wn:0} (b). However, since its mask only depends on $|\Delta f|^2$, it is very sensitive to gaussian noise. That is why, when applying gaussian noise to the input image, Figure \ref{fig:experiments:method:cmp:wn:0.1} (b), the quality decrease quickly. It is interesting to remark that, without noise, the ``L2Insta'' method gives lower $L^2$-error than the ``H1'' one. The ``L2'' method (c) and the ``L2Sta'', use the same mask. However, in the ``L2Sta'' case, we use the time-dependent inpainting which leads to a more pleasant visual result. Indeed, the time-dependent inpainting allows a biggest diffusion of mask's pixels and, as a result, fills the black gap between mask's pixels. The ``L2Inc'' method gives also great results. It might be due to the fact that for each iteration, we add to the mask a small amount of best pixels for the current iteration. Thus, we keep important pixels for previous iterations. Contrasting with the ``L2Inc'' method, the ``L2Dec'' method remove, for each iteration, a small amount of inadequate pixels for the current iteration. Thus it is possible to remove important pixels for previous iterations. It is the same for the ``L2Insta'' method, but it is possible to remove all significant pixels for previous iterations. To conclude, our proposed methods seems to be more robust to gaussian noise than methods proposed in the literature.

\subsection{Salt and Pepper Noise}

Now, we confront our methods to salt and pepper noise. Figure \ref{fig:experiments:method:saltpepper:0.01} are experiments with $1\%$ of salt noise for the first row, and with $1\%$ of pepper noise for the second one.

\begin{figure}[H]
	\centering
	\subfloat[Original image with $1\%$ of salt noise.]{
		\includegraphics[height=3.5cm]{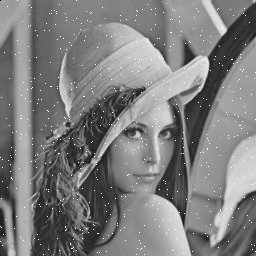}
	}
	\qquad
	\subfloat[Reconstruction using method ``L2Sta''.  $\|u-f_\text{n}\|_{L^2(D)}=25.33$, $\sharp K=6477$.]{
		\includegraphics[height=3.5cm]{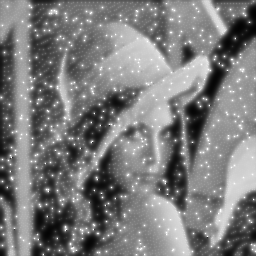}
	}
	\qquad
	\subfloat[Reconstruction using method ``L2Inc''.  $\|u-f_\text{n}\|_{L^2(D)}=12.17$, $\sharp K=6520$.]{
		\includegraphics[height=3.5cm]{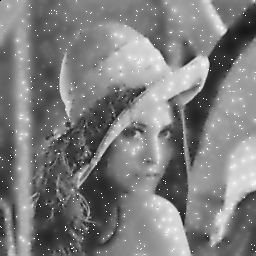}
	}
	\qquad
	\subfloat[Original image with $1\%$ of pepper noise.]{
		\includegraphics[height=3.5cm]{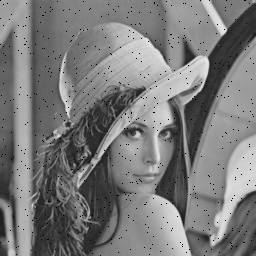}
	}
	\qquad
	\subfloat[Reconstruction using method ``L2Sta''.  $\|u-f_\text{n}\|_{L^2(D)}=22.78$, $\sharp K=6479$.]{
		\includegraphics[height=3.5cm]{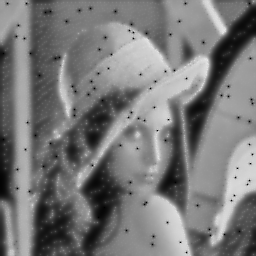}
	}
	\qquad
	\subfloat[Reconstruction using method ``L2Inc''.  $\|u-f_\text{n}\|_{L^2(D)}=11.5$, $\sharp K=6520$.]{
		\includegraphics[height=3.5cm]{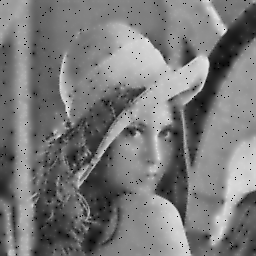}
	}
	
	\caption{Image reconstruction using proposed methods, with $10\%$ of total pixels saved and $1\%$ of salt noise applied to the input image.}
	\label{fig:experiments:method:saltpepper:0.01}
\end{figure} 

Our experiments show that our methods do not give satisfying reconstruction when the input image is significantly corrupted by salt and pepper noise. The diffusion inpainting amplifies the noise by making stains due to noise bigger. This is expected since our method is design to minimize the $L^2$-error which is not suited to remove impulse noise like salt and pepper noise. We suggest to minimize the $L^1$-error instead \cite{ref14, ref15}.

\subsection{Colored Images}

Finally, we propose two strategies for creating masks for colored images. A colored images can be modeled by a function $f$ from $D$ to $[0,1]^3$, $x \mapsto \big(f_\text{R}(x), f_\text{G}(x), f_\text{B}(x)\big)^\text{T}$, where functions $f_\text{R}$, $f_\text{G}$ and $f_\text{B}$ are from $D$ to $\R$, represent red channel, green channel and blue channel respectively. The first strategy consists in creating three masks, one for each channel. This is done in Figure \ref{fig:experiments:method:color} where (a) is the original image, (b) is the mask by keeping $10\%$ of total pixels for each masks and (c) is the reconstructed image. The second strategy is to convert the image into grayscale image, create a mask for the grayscale image and to use it for each channel. This strategy have been applied to Figure \ref{fig:experiments:method:color} (d), (e), by keeping $10\%$ of total pixels, (f), (g), by keeping $15\%$ of total pixels and (h), (i), by keeping $20\%$ of total pixels. In these experiments we used method ``L2Inc'' with $\alpha=8$ and $N=40$.

\begin{figure}[H]
	\centering
	\subfloat[Original image.]{
		\includegraphics[height=3.5cm]{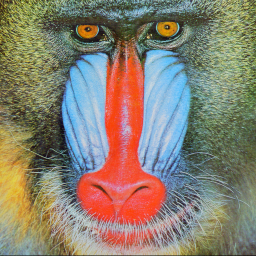}
	}
	\qquad
	\subfloat[Mask for strategy 1 - $10\%$ of pixels per mask.]{
		\includegraphics[height=3.5cm]{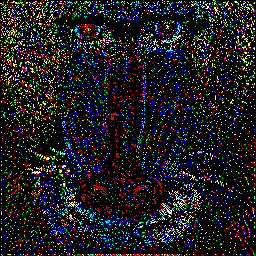}
	}
	\qquad
	\subfloat[Reconstruction using mask for strategy 1.]{
		\includegraphics[height=3.5cm]{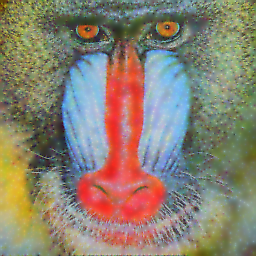}
	}
	\qquad
	\subfloat[Mask for strategy 2 - $10\%$ of pixels.]{
		\includegraphics[height=3.5cm]{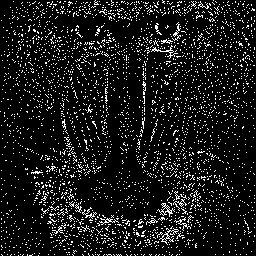}
	}
	\qquad
	\subfloat[Reconstruction using mask (d).]{
		\includegraphics[height=3.5cm]{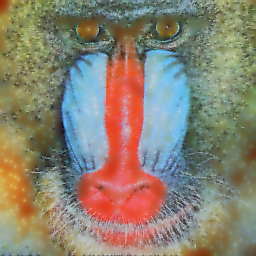}
	}
	\qquad
	\subfloat[Mask for strategy 2 - $15\%$ of pixels.]{
		\includegraphics[height=3.5cm]{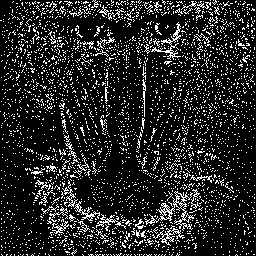}
	}
	\qquad
	\subfloat[Reconstruction using mask (f).]{
		\includegraphics[height=3.5cm]{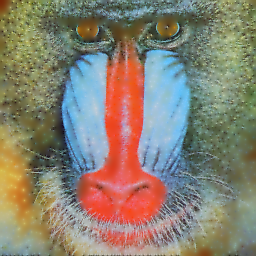}
	}
	\qquad
	\subfloat[Mask for strategy 2 - $20\%$ of pixels.]{
		\includegraphics[height=3.5cm]{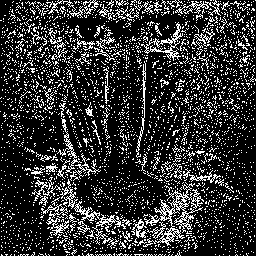}
	}
	\qquad
	\subfloat[Reconstruction using mask (h).]{
		\includegraphics[height=3.5cm]{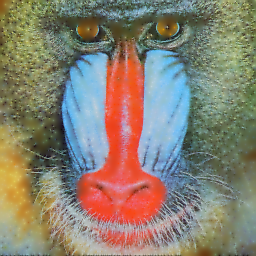}
	}

	\caption{Masks and reconstructions for colored images using ``L2Inc'' method with $\alpha=8$ and $N=40$.}
	\label{fig:experiments:method:color}
\end{figure}

Since in strategy 1 we compute mask with a fixed amount of pixels for each channel, the final mask, where the three masks are combined, may not have the same number of pixels. Indeed, the three masks may not have common pixels or only some common pixels. 
As expected, strategy 1 gives lower $L^2$-error than strategy 2. It can be notice easily be looking at the eyes of the monkey : For strategy 1 they are well-reconstructed but for strategy 2, even with $20\%$ of saved pixels, the reconstruction is not satisfying. An other tempting idea would be to save one mask per color channel, like strategy 1, but with a different amount of pixels for each channel. Indeed, due to color wavelength, the gaussian noise sensitivity is different for channels : the one with large wavelength are less sensitive than the one with thin wavelength.

%% file: tex_parts/conclusion.tex
\section*{Summary and Conclusions}
\label{sec:problem_1:conclusion}
We introduced a mathematical model of the compression problem and its relaxed formulation in the framework of $\Gamma$-convergence. To construct an inpainting mask, we investigated two shape optimization approaches. We obtained a criterion to create a optimal set by using thresholding. The second approach consists in considering ``fat pixels'' instead d of a set of single isolated pixel it yields to a softing of the criteria. We extended our methods to time-dependent problems which result in two-fold improvement of the stationary approach : First, smoothing hard thresholding in the selection for the coding step and giving away to build adaptively the final mask. Secondly, in the decoding phase, it performs the denoising of the input image. We introduced and implemented several algorithms to obtain an optimal mask and, for reconstruction phase, this induces a higher visual result than the stationary method. Numerical experiments confirm that our methods are highly efficient when images are corrupted by gaussian noise. Furthermore, they suggest to favor the density models for the stationary problem and the increasing mask strategy, that we called ``L2Inc'', for the time dependent problems. We confront our methods to salt and pepper, but they do not give satisfying reconstruction when the input image is significantly corrupted by such noise. This was expected since minimizing the $L^2$-error is not the best way to remove impulse noise like salt and pepper noise. We extended the numerical experiments to color images following two procedures The first one consists in creating three masks, one for each channel, while the second creates a mask for the grayscale image and use this mask for each channel. The first approach gives better reconstruction.




%% file: tex_parts/appendix05.tex
\section{Proof of the Estimate of $\mathbf{\theta}$}
\label{appendix:theta}

We aim to give some estimate of the function $\theta$ defined in Theorem \ref{thm:g-convergence}. In the sequel, we set $t_1 := \frac{\sqrt{2}}{2}$. We will widely use the following maximum principle of Problem \ref{pb:problem_1:v} for the proofs. \\

\begin{theorem}[Weak maximum principle] Let us assume that the solution of Problem \ref{pb:problem_1:v} $v_K$ is in $C^2(D)\cap C^0(\bar{D})$.
	If $g\geq 0$ in $D\setminus K$, then $v_K\geq 0$ in $D$.
	\label{prop:problem_1:v:maxprinciple}
\end{theorem}
\begin{proof}
	See \cite{evans10} Theorem 2 in Section 6.4.
\end{proof}

Moreover, we will need the following properties : \\

\begin{lemma} If $v_K$ is the solution of Problem \ref{pb:problem_1:v}, then $v_K$ satisfies
	\[ \|v_K\|_{L^2(D)} \leq \big(1+\alpha C(D)\big)^{-1}\|g\|_{L^2(D)} \]
	and

	\[ \|v_K\|_{L^1(D)} \leq |D|^{1/2}(1+\alpha C(D)\big)^{-1} \|g\|_{L^2(D)}. \]
	\label{prop:problem_1:solution_bounded}
\end{lemma}
\begin{proof}
	Using the weak formulation and Poincaré inequality, we get
	\[ \|v_K\|_{L^2(D)}^2 = \int_D v_K^2\ dx = \int_D gv_K\ dx - \alpha\int_D |\nabla v_K|^2\ dx \leq \int_D gv_K\ dx - \alpha C(D)\int_D v_K^2\ dx.\]

	Hölder inequality gives us

	\[ \|v_K\|_{L^2(D)}^2 \leq \|g\|_{L^2(D)}\|v_K\|_{L^2(D)} - \alpha C(D)\|v_K\|_{L^2(D)}^2. \]

	\[ \|v_K\|_{L^2(D)} \leq \|g\|_{L^2(D)} - \alpha C(D)\|v_K\|_{L^2(D)} \Leftrightarrow \|v_K\|_{L^2(D)} \leq \big(1+\alpha C(D)\big)^{-1}\|g\|_{L^2(D)} \]

	Again, using Hölder inequality 

	\[ \|v_K\|_{L^1(D)} \leq |D|^{1/2}\ \|v_K\|_{L^2(D)} \leq |D|^{1/2}(1+\alpha C(D)\big)^{-1} \|g\|_{L^2(D)}.\]
\end{proof}

\begin{note}

	If $D\subset\R^2$ is convex we have, according to \cite{Payne1960},

	\[ \|v_K\|_{L^2(D)} \leq \big(1+\alpha\pi^2 \diam(D)^{-2}\big)^{-1}\|g\|_{L^2(D)} . \]

	\[ \|v_K\|_{L^1(D)} \leq |D|^{1/2}(1+\alpha\pi^2 \diam(D)^{-2}\big)^{-1} \|g\|_{L^2(D)}. \]

	Moreover, if $g=1$, we have $\|v_K\|_{L^1(D)} \leq (1+\alpha\pi^2 \diam(D)^{-2}\big)^{-1}|D|$.
\end{note}\vspace{0.2cm}

\begin{lemma}
	We have, for $m$ in $(0,t_1)$, \[ \theta(m) \leq C_1(\alpha)\ln{m^{-1}} + C_2(\alpha), \]
	where $C_1$ and $C_2$ are constants depending on $\alpha$.
	\label{prop:problem_1:upper_bound}
\end{lemma}
\begin{proof}
	We consider a particular family of sets $K_n$ in $\mathcal{A}_{m,n}$. We choose an integer $k$ such that $n = k^2$ and we suppose $K_n\in\mathcal{A}_{m,n}$ are composed of $n$ balls of radius $m/k$, with their centers superposing the centers of the $k^2$ squares of side $1/k$ of a regular lattice partitioning the square $I^2$.

	\begin{figure}[!ht]
		\centering
		\subfloat[$K_1$]{
			\includegraphics[width=3.5cm]{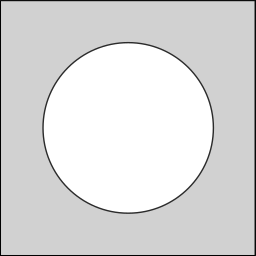}
		}
		\qquad
	    \subfloat[$K_4$]{
			\includegraphics[width=3.5cm]{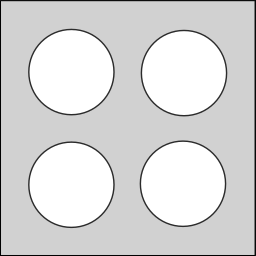}
		}
		\qquad
		\subfloat[$K_9$]{
			\includegraphics[width=3.5cm]{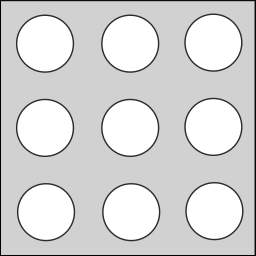}
		}
		\caption{Drawing of $I^2\setminus K_n$ with $n := k^2$, for $k=1,2,3$.}
	\end{figure} 

	Let us denote $v^1_{K_n}$ the solution of Problem \ref{pb:problem_1:v} with $g=1$, $K=K_n$ and $D=I^2$. It holds $\int_I v^1_{K_1}\ dx = n\int_I v^1_{K_n}\ dx$. We recall 

	\[ \theta(m) := \inf_{K_n\in\mathcal{A}_{m,n}} \liminf_n n \int_D gv_{K_n}\ dx. \]

	In particular \[ \theta(m) \leq \liminf_n n \int_D v^1_{K_n}\ dx = \int_I v^1_{K_1}\ dx. \]

	We have $I^2\subset B(x_0,t_1)$. Therefore if we denote by $w$ the solution of Problem \ref{pb:problem_1:v} with $g=1$, $D=B(x_0,t_1)$ and $K=K_1:=B(x_0,m)$, it holds by the maximum principle that $v^1_{K_1} \leq w$. Then we have the following estimate 

	\[ \theta(m) \leq \int_{B(x_0,t_1)} w\ dx. \]

	Let us consider the following problem

	\begin{equation}
		\left\{\begin{array}{rl}
			-\alpha\Delta \tilde{w} = 1 & \text{in}\ B(x_0,t_1)\setminus\overline{B(x_0,m)}, \\
			\tilde{w} = 0 & \text{in}\ \overline{B(x_0,m)}, \\
			\frac{\partial \tilde{w}}{\partial \mathbf{n}} = 0 & \text{on}\ \partial B(x_0,t_1).
		\end{array}\right .
	\end{equation}

	Now, we set $e:=w-\tilde{w}$. Thus we have for all $v$ in $H^1_0(B(x_0,t_1)\setminus\overline{B(x_0,m)})$

	\[ \alpha\int_{B}\nabla e\cdot\nabla v\ dx = \alpha\int_{B}\nabla w\cdot\nabla v\ dx - \alpha\int_{B}\nabla \tilde{w}\cdot\nabla v\ dx = -\int_{B}wv\ dx, \]

	i.e. $e$ satisfies the problem below

	\begin{equation}
		\left\{\begin{array}{rl}
			-\alpha\Delta e = -w & \text{in}\ B(x_0,t_1)\setminus\overline{B(x_0,m)}, \\
			e = 0 & \text{in}\ \overline{B(x_0,m)}, \\
			\frac{\partial e}{\partial \mathbf{n}} = 0 & \text{on}\ \partial B(x_0,t_1).
		\end{array}\right .
	\end{equation}

	Since $w\geq 0$, we have by the maximum principle that $e\leq 0$ i.e. $0\leq w\leq \tilde{w}$. By consequence

	\[ \theta(m) \leq \int_{B(x_0,t_1)} \tilde{w}\ dx. \]

	The solution $\tilde{w}$ have been computed in \cite{ref7}. Due to the radial symmetry of $\tilde{w}$, we can get explicitly $\tilde{w}$, solution of :

	\begin{equation}
		\left\{\begin{array}{rl}
			\tilde{w}''(r) + \frac{1}{r} \tilde{w}'(r) = -\frac{1}{\alpha} & \text{if}\ m<r< t_1, \\
			\tilde{w} = 0 & \text{if}\ 0\leq r\leq m, \\
			\tilde{w}'(t_1) = 0 & .
		\end{array}\right .
	\end{equation}

	For $r=|x-x_0|$ we have

	\[ \tilde{w}(x) = \begin{cases}
		k\ln(\frac{r}{m}) - \frac{1}{4\alpha}(r^2-m^2) & \text{if}\ m<r< t_1 \\ 
		0 & \text{if}\ 0\leq r\leq m \\

	\end{cases},\ k=\frac{m t_1^2}{2\alpha}. \]

	Integrating $\tilde{w}$ over $B(x_0,t_1)$

	\begin{align*}	
		\int_{B(x_0,t_1)} \tilde{w}\ dx & = 2\pi\int_m^{t_1} \Big(k\ln(\frac{r}{m}) - \frac{1}{4\alpha}(r^2-m^2)\Big)r\ dr \\
		& = 2\pi k\int_m^{t_1} r\ln(\frac{r}{m})\ dr - \frac{\pi}{2\alpha}\int_m^{t_1} r^3\ dr + \frac{\pi}{2\alpha}m^2\int_m^{t_1} r\ dr \\
		& = \pi kt_1^2\ln{\frac{t_1}{m}} + \frac{\pi}{2}\underbrace{(\frac{m^2}{2\alpha} - \frac{k}{m})}_{\leq\ 0}(t_1^2 - m^2) + \frac{\pi}{8\alpha}\underbrace{(m^4-t_1^4)}_{\leq\ 0} \\
		& \leq \frac{\pi t_1^4}{2\alpha}m \ln{\frac{t_1}{m}} \leq \frac{\pi t_1^4}{2\alpha}\ln{\frac{t_1}{m}} = \underbrace{\frac{\pi t_1^4}{2\alpha}}_{=:\ C_1(\alpha)}\ln{m^{-1}} + \underbrace{\frac{\pi t_1^4}{2\alpha}\ln{t_1}}_{=:\ C_2(\alpha)}.
	\end{align*}
\end{proof} \vspace{0.2cm}

\begin{lemma}
	We have, for $m$ in $(0,t_1)$, \[ C_1(\alpha)\ln(m^{-1}) - C_2(\alpha) \leq \theta(m),\]

	where $C_1$ and $C_2$ are constants depending on $\alpha$.

	\label{prop:problem_1:lower_bound}
\end{lemma}
\begin{proof}
	We fix $n\in\N$ and $(x_i)_{i=1,\dots, n}\in I^2$. We consider the sets $K_m := \bigcup_{i=1}^n \overline{B(x_i, m n^{-1/2})}$ and $U_m := I^2\setminus K_m$. Let us denote $u_m^1$ the solution of Problem \ref{pb:problem_1:v} with $g=1$, $K=K_m$ and $D=I^2$. Using Holder inequality we have

	\[ \Big(\int_{\partial U_m}  \frac{\partial u_m^1}{\partial n}d\mathcal{H}^1\Big)^2 \leq \int_{\partial U_m} \Big|\frac{\partial u_m^1}{\partial n}\Big|^2\ d\mathcal{H}^1 \times \int_{\partial U_m}  \ d\mathcal{H}^1 = \mathcal{H}^1(\partial U_m) \int_{\partial U_m} \Big|\frac{\partial u_m^1}{\partial n}\Big|^2\ d\mathcal{H}^1. \]

	Moreover, we have thanks to the Green formula

	\[ \int_{\partial U_m} \frac{\partial u_m^1}{\partial n}d\mathcal{H}^1 = \int_{U_m}  \Delta u_m^1\ dx = \frac{1}{\alpha}\int_{U_m}  (u_m^1-1)\ dx = \frac{1}{\alpha}\|u_m^1\|_{L^1(U_m)} - \frac{1}{\alpha}|U_m|. \]

	Thus 
	
	\begin{align*}
	    & \Big(\frac{1}{\alpha}\|u_m^1\|_{L^1(U_m)} - \frac{1}{\alpha}|U_m|\Big)^2 \leq \mathcal{H}^1(\partial U_m) \int_{\partial U_m} \Big|\frac{\partial u_m^1}{\partial n}\Big|^2\ d\mathcal{H}^1 \\
	    \Leftrightarrow & \frac{1}{\alpha^2}|U_m|^2 - \frac{2}{\alpha^2}\|u_m^1\|_{L^1(U_m)}|U_m| \leq \mathcal{H}^1(\partial U_m) \int_{\partial U_m} \Big|\frac{\partial u_m^1}{\partial n}\Big|^2\ d\mathcal{H}^1.
	\end{align*}

	The fact that $|U_m| \geq 1-2\pi m^2$ and Property \ref{prop:problem_1:solution_bounded} give us 

	\[ \frac{2\pi^2\alpha - 1}{\alpha^2(1+2\pi^2\alpha)} - \frac{2\pi}{\alpha^2}m \leq \mathcal{H}^1(\partial U_m) \int_{\partial U_m} \Big|\frac{\partial u_m^1}{\partial n}\Big|^2\ d\mathcal{H}^1. \]

	Also, it holds $\mathcal{H}^1(\partial U_m) \leq 2\pi m\sqrt{n}$. Then 
	
	\begin{align*}
	    & \frac{2\pi^2\alpha - 1}{\alpha^2(1+2\pi^2\alpha)} - \frac{2\pi}{\alpha^2}m \leq 2\pi m\sqrt{n}\int_{\partial U_m} \Big|\frac{\partial u_m^1}{\partial n}\Big|^2\ d\mathcal{H}^1 \\
	    \Leftrightarrow & \frac{2\pi^2\alpha - 1}{2\pi\alpha^2(1+2\pi^2\alpha)}\frac{1}{m} - \frac{1}{\alpha^2} \leq \sqrt{n}\int_{\partial U_m} \Big|\frac{\partial u_m^1}{\partial n}\Big|^2\ d\mathcal{H}^1.
	\end{align*}

	Using that $-\frac{d F}{dm} = \sqrt{n}^{-1} \int_{\partial U_m} \Big|\frac{\partial u_m^1}{\partial n}\Big|^2\ d\mathcal{H}^1$, we have

	\[ \frac{2\pi^2\alpha - 1}{2\pi\alpha^2(1+2\pi^2\alpha)}\frac{1}{m} - \frac{1}{\alpha^2} \leq - n \frac{d F}{dm}. \]

	Integrating over $[m_1,m_2] \subset (0,t_1)$ yields to

	\[ \frac{2\pi^2\alpha - 1}{2\pi\alpha^2(1+2\pi^2\alpha)}\ln\Big(\frac{m_2}{m_1}\Big) - \frac{m_2-m_1}{\alpha^2} + nF_{m_2} \leq nF_{m_1}. \]

	Taking inf over $x_i$ and passing to $\liminf$ over $n$ when $n$ tends to $+\infty$ leads to

	\[ \frac{2\pi^2\alpha - 1}{2\pi\alpha^2(1+2\pi^2\alpha)}\ln\Big(\frac{m_2}{m_1}\Big) - \frac{m_2-m_1}{\alpha^2} + \theta(m_2) \leq \theta(m_1). \]

	In particular, if $m_2 = t_1$ and $m_1 = m$, $0<m<t_1 = \frac{\sqrt{2}}{2}$,
	
	\begin{align*}
	    & \frac{2\pi^2\alpha - 1}{2\pi\alpha^2(1+2\pi^2\alpha)}\ln\Big(\frac{t_1}{m}\Big) - \frac{t_1-m}{\alpha^2} \leq \theta(m) \\
	    \Leftrightarrow & \frac{2\pi^2\alpha - 1}{2\pi\alpha^2(1+2\pi^2\alpha)}\ln\Big(\frac{t_1}{m}\Big) - \frac{t_1}{\alpha^2} \leq \theta(m) \\
	    \Leftrightarrow & \underbrace{\frac{2\pi^2\alpha - 1}{2\pi\alpha^2(1+2\pi^2\alpha)}}_{=:\ C_1(\alpha)}\ln(m^{-1}) - \underbrace{\Big(\frac{2\pi^2\alpha - 1}{2\pi\alpha^2(1+2\pi^2\alpha)}\ln(t_1^{-1}) + \frac{t_1}{\alpha^2}\Big)}_{=:\ C_2(\alpha)} \leq \theta(m).
	\end{align*}
\end{proof}

%% file: tex_parts/appendix01.tex
\section{Asymptotic Development Calculus}

Let $x_0$ be in $\R^2$ and $\varepsilon>0$. In this section, we aim to find an estimate of

\[ \int_{B(x_0,\varepsilon)}w\ dx, \]

where $w$ the solution of the problem below : \\

\begin{problem} Find $w$ in $H^1_0\big(B(x_0,\varepsilon)\big)$ such that
	\begin{equation}
		\left\{\begin{array}{rl}
			w - \alpha \Delta w = g, & \text{in}\ B(x_0,\varepsilon), \\
			w = 0, & \text{on}\ \partial B(x_0,\varepsilon).
		\end{array}\right .
	\end{equation}
	\label{pb:appendix:nonhomogeneous:l2}
\end{problem}

We did not find this result in the literature despite it may exist. For the sake of completeness, we propose a way to find this estimate. To solve Problem \ref{pb:appendix:nonhomogeneous:l2}, we use Green functions $G:B(x_0,\varepsilon)\times B(x_0,\varepsilon)$, corresponding to Problem \ref{pb:appendix:nonhomogeneous:l2} which are solution to \\

\begin{problem} Find $G(\cdot,y)$ in $H^1_0\big(B(x_0,\varepsilon)\big)$ such that
    \begin{equation}
    	\left\{\begin{array}{rl}
    		G(x,y) - \alpha \Delta_x G(x,y) = \delta_y(x), & x\in B(x_0,\varepsilon), \\
    		G(x,y) = 0, & x\in\partial B(x_0,\varepsilon),
    	\end{array}\right .
    \end{equation}
    for $y$ in $B(x_0,\varepsilon)$.
    \label{pb:appendix:nonhomogeneous:l2:Green}
\end{problem}

We have, \\

\begin{proposition}
    Let $G$ be Green functions corresponding to Problem \ref{pb:appendix:nonhomogeneous:l2:Green}. Then, for $x$ in $\overline{B(x_0,\varepsilon)}$,
    
    \[ w(x) := \int_{B(x_0,\varepsilon)} g(y)G(x,y)\ dy, \]
    
    is the solution of Problem \ref{pb:appendix:nonhomogeneous:l2}.
    \label{prop:appendix:int_calculus}
\end{proposition}
\begin{proof}
    Let $x$ be in $B(x_0,\varepsilon)$.
    
    \begin{align*}
        w(x) - \alpha \Delta w(x) &= \int_{B(x_0,\varepsilon)} g(y)G(x,y)\ dy - \alpha \int_{B(x_0,\varepsilon)} g(y)\Delta_x G(x,y)\ dy \\
            &= \int_{B(x_0,\varepsilon)} g(y)\big(G(x,y) - \alpha\Delta_x G(x,y)\big)\ dy \\
            &= \int_{B(x_0,\varepsilon)} g(y)\delta_y(x)\ dy \\
            &= g(x).
    \end{align*}
    
    Moreover, we have,  $w(x)=0$, for $x$ on $\partial B(x_0,\varepsilon)$.
\end{proof}

From now, our goal is to find Green functions $G$. To do so, we write $G$ as the sum of a particular solution $G_\text{p}$ of Problem \ref{pb:appendix:nonhomogeneous:l2:Green} without the boundary condition, and the general solution $G_0$ of the homogeneous version of Problem \ref{pb:appendix:nonhomogeneous:l2:Green} such that $G_0 = -G_\text{p}$ on $\partial B(x_0,\varepsilon)$. Below is the main proposition of this section, \\

\begin{proposition} We have when $\varepsilon$ tends to 0,
    \[ \int_{B(x_0,\varepsilon)} w(x)\ dx = -g(x_0)\pi\varepsilon^4\ln(\varepsilon) + O(\varepsilon^4).\]
    \label{prop:appendix:int_calculus:main}
\end{proposition}
\begin{proof}
    For $\varepsilon$ small enough, we have, using Proposition \ref{prop:appendix:int_calculus},
    
    \[ \int_{B(x_0,\varepsilon)} w(x)\ dx = \int_{B(x_0,\varepsilon)} \int_{B(x_0,\varepsilon)} g(y)G(x,y)\ dy\ dx. \]
        
    Using Fubini, we get
    
    \begin{align*}
        \int_{B(x_0,\varepsilon)} w(x)\ dx & = \int_{B(x_0,\varepsilon)} g(y) \int_{B(x_0,\varepsilon)}G(x,y)\ dx\ dy \\
        & = \int_{B(x_0,\varepsilon)} g(y) \int_{B(x_0,\varepsilon)}G_\text{p}(x,y)\ dx\ dy + \int_{B(x_0,\varepsilon)} g(y) \int_{B(x_0,\varepsilon)}G_0(x,y)\ dx\ dy.
    \end{align*}

    Proposition \ref{prop:appendix:int_calculus:1} and Proposition \ref{prop:appendix:int_calculus:2} give us the result.
\end{proof}

It remains to state and to prove Proposition \ref{prop:appendix:int_calculus:1} and Proposition \ref{prop:appendix:int_calculus:2}. We start by giving an explicit expression for $G_\text{p}$ with the following proposition. \\

\begin{proposition} For $x$ and $y$ in $B(x_0,\varepsilon)$ such that $x\neq y$, we have
    \[ G_\text{p}(x,y) = \frac{1}{2\pi}K_0\Big(\frac{1}{\sqrt{\alpha}}|x-y|\Big), \]
    
    where $K_0$ is the modified Bessel function of the second kind, see \cite{atlasfunctions}.
\end{proposition}
\begin{proof}
    See \cite{ref16}.
\end{proof}

Then, we compute the first part of Proposition \ref{prop:appendix:int_calculus:main}. \\

\begin{proposition} When $\varepsilon$ tends to $0$, we have,
    \[ \int_{B(x_0,\varepsilon)} g(y) \int_{B(x_0,\varepsilon)}G_\text{p}(x,y)\ dx\ dy = -g(x_0)\frac{\pi}{2}\varepsilon^4\ln(\varepsilon) + O(\varepsilon^4). \]
    \label{prop:appendix:int_calculus:1}
\end{proposition}
\begin{proof}
    According to \cite{atlasfunctions}, we have the following asymptotic development
    
    \[ K_0(z) = -\ln z + \ln 2 - \gamma + O(z^2|\ln z|), \]
    
    for $z\to 0$, where $\gamma$ denotes the Euler–Mascheroni constant. Then,
    
    \[ G_\text{p}(x,y) = -\frac{1}{2\pi}\Big(\ln|x-y| + \frac{1}{2}\ln \alpha + \ln 2 - \gamma\Big) + O(|x-y|^2|\ln|x-y||), \]
    
    for $|x-y|\to 0$. Thus, 
    
    \begin{align*}
        & \int_{B(x_0,\varepsilon)} g(y) \int_{B(x_0,\varepsilon)}G_\text{p}(x,y)\ dx\ dy = -\frac{1}{2\pi}\int_{B(x_0,\varepsilon)} g(y) \int_{B(x_0,\varepsilon)}\ln|x-y|\ dx\ dy \\
        & \hspace{1cm} -\Big(\frac{1}{2}\ln \alpha + \ln 2 - \gamma\Big)\frac{\varepsilon^2}{2}\int_{B(x_0,\varepsilon)} g(y) \ dy + O(1)\int_{B(x_0,\varepsilon)} g(y) \int_{B(x_0,\varepsilon)}|x-y|^2|\ln|x-y||\ dx\ dy.
    \end{align*}
    
    Using Taylor's formula, we have 
    
    \begin{align*}
        & \int_{B(x_0,\varepsilon)} g(y) \int_{B(x_0,\varepsilon)}G_\text{p}(x,y)\ dx\ dy = -\frac{1}{2\pi}g(x_0)\int_{B(x_0,\varepsilon)}  \int_{B(x_0,\varepsilon)}\ln|x-y|\ dx\ dy \\
        & \hspace{1cm} + O(1)\int_{B(x_0,\varepsilon)} \|y-x_0\| \int_{B(x_0,\varepsilon)}\ln|x-y|\ dx\ dy -\Big(\frac{1}{2}\ln \alpha + \ln 2 - \gamma\Big)\frac{\pi}{2}g(x_0)\varepsilon^4 + O(\varepsilon^4) \\
        & \hspace{1cm} + O(1)\int_{B(x_0,\varepsilon)}\int_{B(x_0,\varepsilon)}|x-y|^2|\ln|x-y||\ dx\ dy + O(1)\int_{B(x_0,\varepsilon)} \|y-x_0\| \int_{B(x_0,\varepsilon)}|x-y|^2|\ln|x-y||\ dx\ dy \\
        \leq & -\frac{1}{2\pi}g(x_0)\int_{B(x_0,\varepsilon)}  \int_{B(x_0,\varepsilon)}\ln(2\varepsilon)\ dx\ dy + O(1)\int_{B(x_0,\varepsilon)} \|y-x_0\| \int_{B(x_0,\varepsilon)}\ln(2\varepsilon)\ dx\ dy \\
        & \hspace{1cm} -\Big(\frac{1}{2}\ln \alpha + \ln 2 - \gamma\Big)\frac{\pi}{2}g(x_0)\varepsilon^4 + O(\varepsilon^4) \\
        & \hspace{1cm} + O(1)\int_{B(x_0,\varepsilon)}\int_{B(x_0,\varepsilon)}|x-y|^2|\ln(2\varepsilon)|\ dx\ dy + O(1)\int_{B(x_0,\varepsilon)} \|y-x_0\| \int_{B(x_0,\varepsilon)}|x-y|^2|\ln(2\varepsilon)|\ dx\ dy \\
        = & -\frac{\pi}{2}g(x_0)\varepsilon^4\ln(2\varepsilon) + O\Big(\varepsilon^5\ln(2\varepsilon)\Big) -\Big(\frac{1}{2}\ln \alpha + \ln 2 - \gamma\Big)\frac{\pi}{2}g(x_0)\varepsilon^4 + O(\varepsilon^4) + O\Big(\varepsilon^6\ln(2\varepsilon)\Big) + O\Big(\varepsilon^7\ln(2\varepsilon)\Big) \\
        = & -\frac{\pi}{2}g(x_0)\varepsilon^4\ln(\varepsilon) + O\Big(\varepsilon^5\ln(\varepsilon)\Big) -\Big(\frac{1}{2}\ln \alpha - \gamma\Big)\frac{\pi}{2}g(x_0)\varepsilon^4 + O(\varepsilon^4) + O\Big(\varepsilon^6\ln(\varepsilon)\Big) + O\Big(\varepsilon^7\ln(\varepsilon)\Big) \\
        = & -\frac{\pi}{2}g(x_0)\varepsilon^4\ln(\varepsilon) + O(\varepsilon^4).
    \end{align*}
\end{proof}

And we finish by computing the second part of Proposition \ref{prop:appendix:int_calculus:main}. \\

\begin{proposition} When $\varepsilon$ tends to $0$, we have,
    \[ \int_{B(x_0,\varepsilon)} g(y) \int_{B(x_0,\varepsilon)}G_0(x,y)\ dx\ dy = -g(x_0) \frac{\pi}{2}\varepsilon^4\ln \varepsilon + O(\varepsilon^4). \]
    \label{prop:appendix:int_calculus:2}
\end{proposition}
\begin{proof} We start by using Taylor's formula on $g$ around $x_0$,

    \begin{align*}
        & \int_{B(x_0,\varepsilon)} g(y) \int_{B(x_0,\varepsilon)}G_0(x,y)\ dx\ dy = g(x_0) \int_{B(x_0,\varepsilon)}\int_{B(x_0,\varepsilon)}G_0(x,y)\ dx\ dy \\
        & \hspace{1cm} + O(1)\int_{B(x_0,\varepsilon)}\|y-x_0\| \int_{B(x_0,\varepsilon)}G_0(x,y)\ dx\ dy \\
        \leq & g(x_0) \pi\varepsilon^2 \int_{B(x_0,\varepsilon)}\|G_0(\cdot,y)\|_{L^\infty\big(B(x_0,\varepsilon)\big)}\ dy + O(\varepsilon^2)\int_{B(x_0,\varepsilon)}\|y-x_0\| \|G_0(\cdot,y)\|_{L^\infty\big(B(x_0,\varepsilon)\big)}\ dy.
    \end{align*}
    
    Since $G_0$ satisfies the maximum principle \cite{evans10},
    
    \begin{align*}
        & \int_{B(x_0,\varepsilon)} g(y) \int_{B(x_0,\varepsilon)}G_0(x,y)\ dx\ dy \leq g(x_0) \pi\varepsilon^2 \int_{B(x_0,\varepsilon)}\|G_0(\cdot,y)\|_{L^\infty\big(\partial B(x_0,\varepsilon)\big)}\ dy \\
        & \hspace{1cm} + O(\varepsilon^2)\int_{B(x_0,\varepsilon)}\|y-x_0\| \|G_0(\cdot,y)\|_{L^\infty\big(\partial B(x_0,\varepsilon)\big)}\ dy \\
        = & g(x_0) \pi\varepsilon^2 \int_{B(x_0,\varepsilon)}\|G_\text{p}(\cdot,y)\|_{L^\infty\big(\partial B(x_0,\varepsilon)\big)}\ dy + O(\varepsilon^2)\int_{B(x_0,\varepsilon)}\|y-x_0\| \|G_\text{p}(\cdot,y)\|_{L^\infty\big(\partial B(x_0,\varepsilon)\big)}\ dy \\
        = & g(x_0) \frac{1}{2}\varepsilon^2 \int_{B(x_0,\varepsilon)}\Big\|K_0\Big(\frac{1}{\sqrt{\alpha}}|\cdot-y|\Big)\Big\|_{L^\infty\big(\partial B(x_0,\varepsilon)\big)}\ dy + O(\varepsilon^2)\int_{B(x_0,\varepsilon)}\|y-x_0\| \Big\|K_0\Big(\frac{1}{\sqrt{\alpha}}|\cdot-y|\Big)\Big\|_{L^\infty\big(\partial B(x_0,\varepsilon)\big)}\ dy.
    \end{align*}

    According to \cite{atlasfunctions}, $K_0$ is an increasing function. Thus, for $y$ in $B(x_0,\varepsilon)$,
    
    \[ \Big\|K_0\Big(\frac{1}{\sqrt{\alpha}}|\cdot-y|\Big)\Big\|_{L^\infty\big(\partial B(x_0,\varepsilon)\big)} := \sup_{x\in\partial B(x_0,\varepsilon)} \Big|K_0\Big(\frac{1}{\sqrt{\alpha}}|x-y|\Big)\Big|, \]
    
    is attained where $|x-y| := \sqrt{r_x^2 + r_y^2 - 2r_x r_y\cos{(\theta_x-\theta_y)}}$ is maximal, i.e. when $\cos{(\theta_x-\theta_y)} = -1$, i.e. for $\theta_x=\pi+\theta_y$. In that case,
    
    \[ |x-y| = \sqrt{\varepsilon^2 + r_y^2 + 2\varepsilon r_y} = \varepsilon + r_y. \]
    
    Thus,
    
    \[ \|G_\text{p}(\cdot,y)\|_{L^\infty\big(\partial B(x_0,\varepsilon)\big)} = \frac{1}{2\pi} K_0\Big(\frac{1}{\sqrt{\alpha}}(\varepsilon+r_y)\Big). \]
    
    Then, we have
    
    \begin{align*}
         \int_{B(x_0,\varepsilon)} g(y) \int_{B(x_0,\varepsilon)}G_0(x,y)\ dx\ dy &\leq g(x_0) \pi \varepsilon^2 \int_0^\varepsilon r_y K_0\Big(\frac{1}{\sqrt{\alpha}}(\varepsilon+r_y)\Big)\ dy + O(\varepsilon^2)\int_0^\varepsilon r_y^2 K_0\Big(\frac{1}{\sqrt{\alpha}}(\varepsilon+r_y)\Big)\ dy \\
         & \leq g(x_0) \frac{\pi}{2}\varepsilon^4 K_0\Big(\frac{2\varepsilon}{\sqrt{\alpha}}\Big) + O(\varepsilon^5)K_0\Big(\frac{2\varepsilon}{\sqrt{\alpha}}\Big).
    \end{align*}

    Again, we use that, when $z$ tends to $0$,
    
    \[ K_0(z) = -\ln z + \ln 2 - \gamma + O(z^2|\ln z|), \]
    
    and get, since $\varepsilon$ tends to $0$,
    
    \[ K_0\Big(\frac{2\varepsilon}{\sqrt{\alpha}}\Big) = -\ln \varepsilon + \frac{1}{2}\ln \alpha - \gamma + O(\varepsilon^2|\ln \varepsilon|). \]
    
    Therefore,

    \begin{align*}
        \int_{B(x_0,\varepsilon)} g(y) \int_{B(x_0,\varepsilon)}G_0(x,y)\ dx\ dy &\leq -g(x_0) \frac{\pi}{2}\varepsilon^4\ln \varepsilon + g(x_0) \frac{\pi}{4}\varepsilon^4\ln \alpha - g(x_0) \frac{\pi}{2}\varepsilon^4\gamma + O(\varepsilon^5\ln \varepsilon) \\
        &= -g(x_0) \frac{\pi}{2}\varepsilon^4\ln \varepsilon + O(\varepsilon^4).
    \end{align*}
\end{proof}

%% file: tex_parts/appendix02.tex
\section{Algorithms}
\label{appendix:algos}

Here, we present and discuss algorithms used in Section \ref{sec:problem_1:numerical_results}. Each algorithm is used during the encoding step and gives an inpainting mask $K$ subset of $D$. The input data $f$ is the image to compress, $\delta t$ is the time-step of the parabolic inpainting discretization, Section \ref{sec:time-dependent}, $\alpha > 0$ and $c$ is the percentage amount of pixels in the mask i.e. the percentage amount of pixels to save. We denote in the sequel for a discrete mask $A$ in $D$, $\sharp A$, the number of pixels in $A$.

\subsection{L2Dec Method}

For this method, we want, for $n$ in $\{0,\hdots, N\}$, $K_{n+1}\subset K_n$. The parameter $N$ in $\mathbb{N}^*$ is the number of time-step we want to compute. Thus, we compute the solution $u$ of Problem \ref{pb:problem_1_time} for $t=N\delta t\alpha$.
For each step $n$ in $\{0,\hdots, N\}$, we set $K_n$ as the thresholding of $\mathbf{1}_{K_{n-1}}(f-\delta t \alpha\Delta f-u^n)^2$ such that $\sharp K_n \leq \sharp K_{n-1}$. The term $\mathbf{1}_{K_{n-1}}$ ensure us that $K_n$ is a subset of $K_{n-1}$. Then, we compute $u^{n+1}$, solution of Problem \ref{pb:problem_1_time_implicit} with $u^n$, $f$ and $K_n$.
We compare two choices for $u_0$. We start when $u_0 = 0$. Thus, $K_0$ is the thresholding of $(f-\delta t \alpha\Delta f)^2$. It corresponds to the stationary case, Section \ref{sec:problem_1:topo_grad}. Then, when $u_0 = f$, $K_0$ is the thresholding of $(\Delta f)^2$. In that case, $\delta t$ and $\alpha$ do not have any influence on the output mask $K_N$ given by the method. Indeed, since every $K_n$ is a subset of $K_{n-1}$, the resulting mask will be the same as the $H^1$ one proposed in \cite{ref2} since $u_n = f$ in $K_{n-1}$. Figure \ref{fig:appendix:L2Dec} is results for the two cases when $N=50$ and $c=0.1$.

\begin{figure}[H]
	\centering
	\subfloat[Original image]{
		\includegraphics[height=3.5cm]{resources/images/Lenna.png}
	}
	\qquad
	\subfloat[Mask when $u_0=0$ for $\alpha=1$ and $\delta t=0.1$.]{
		\includegraphics[height=3.5cm]{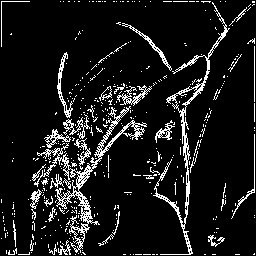}
	}
	\qquad
	\subfloat[Reconstruction $u$ when $u_0=0$. $\|f-u\|_{L^2(D)}=28.15$.]{
		\includegraphics[height=3.5cm]{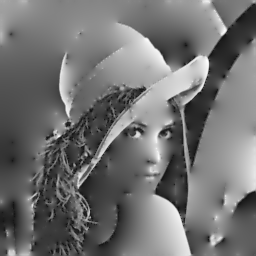}
	}
	\qquad
	\subfloat[Mask when $u_0=f$.]{
		\includegraphics[height=3.5cm]{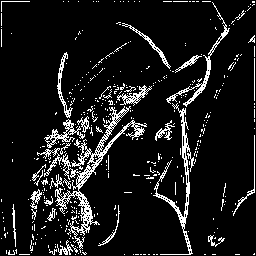}
	}
	\qquad
	\subfloat[Reconstruction $u$ when $u_0=f$. $\|f-u\|_{L^2(D)}=29.69$.]{
		\includegraphics[height=3.5cm]{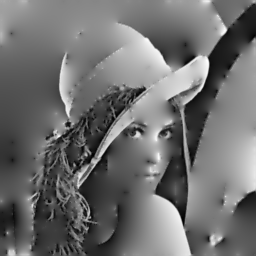}
	}
	
	\caption{Initial condition comparison for ``L2Dec'' method.}
	\label{fig:appendix:L2Dec}
\end{figure}

\subsection{L2Inc Method}

Now, we present the ``L2Inc'' method. Unlike the ``L2Dec'' method, we want $K_{n-1}\subset K_n$ for $n$ in $\{0,\hdots, N\}$, $N$ in $\mathbb{N}^*$.
For each step $n$ in $\{0,\hdots, N\}$, we set $K_n$ as the union of $K_{n-1}$ and the thresholding of $(1-\mathbf{1}_{K_{n-1}})(f-\delta t \alpha\Delta f-u^n)^2$. Therefore, we have $\sharp K_n \geq \sharp K_{n-1}$. The term $(1-\mathbf{1}_{K_{n-1}})$ ensure us to not add to $K_n$ pixels that are already in $K_{n-1}$. Then, we compute $u^{n+1}$, solution of Problem \ref{pb:problem_1_time_implicit} with $u^n$, $f$ and $K_n$.
Again, we have to chose the initial condition $u_0$. We propose and compare two choices for $u_0$. When $u_0 = 0$, $K_0$ corresponds to the stationary case, Section \ref{sec:problem_1:topo_grad}. When $u_0 = f$, $K_0$ is the ``H1'' mask. Figure \ref{fig:appendix:L2Inc} is results for the two cases when $N=40$, $\delta t = 0.1$, $\alpha=20$ and $c=0.1$.

\begin{figure}[H]
	\centering
	\subfloat[Original image]{
		\includegraphics[height=3.5cm]{resources/images/Lenna.png}
	}
	\qquad
	\subfloat[Mask when $u_0=0$.]{
		\includegraphics[height=3.5cm]{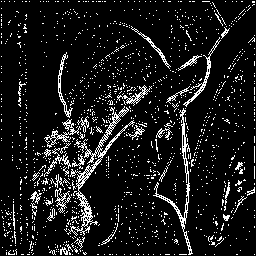}
	}
	\qquad
	\subfloat[Reconstruction $u$ when $u_0=0$. $\|f-u\|_{L^2(D)}=14.48$.]{
		\includegraphics[height=3.5cm]{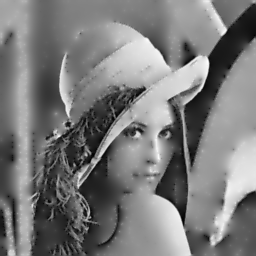}
	}
	\qquad
	\subfloat[Mask when $u_0=f$.]{
		\includegraphics[height=3.5cm]{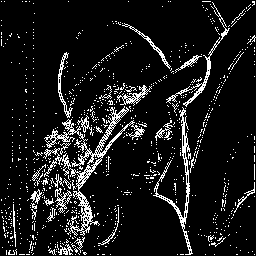}
	}
	\qquad
	\subfloat[Reconstruction $u$ when $u_0=f$. $\|f-u\|_{L^2(D)}=15.89$.]{
		\includegraphics[height=3.5cm]{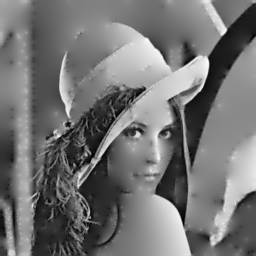}
	}
	
	\caption{Initial condition comparison for ``L2Inc'' method.}
	\label{fig:appendix:L2Inc}
\end{figure}

\subsection{L2Insta Method}

Finally, we propose a last time-dependent method : the ``L2Insta'' method. Here, we do not impose sets to be included into others, but we still want $\sharp K_n = \sharp K_{n-1}$.
For each step $n$ in $\{0,\hdots, N\}$, $K_n$ is the halftoning of $(f-\delta t \alpha\Delta f-u^n)^2$. Then, we compute $u^{n+1}$, solution of Problem \ref{pb:problem_1_time_implicit} with $u^n$, $f$ and $K_n$.
Again, we have to chose the initial condition $u_0$. We propose and compare two choices for $u_0$. When $u_0 = 0$, $K_0$ corresponds to the stationary case, Section \ref{sec:problem_1:topo_grad}. When $u_0 = f$, $K_0$ is the ``H1'' mask. Figure \ref{fig:appendix:l2insta} is results for the two cases when $N=10$, $\delta t = 0.1$, $\alpha=10$ and $c=0.1$.

\begin{figure}[H]
	\centering
	\subfloat[Original image]{
		\includegraphics[height=3.5cm]{resources/images/Lenna.png}
	}
	\qquad
	\subfloat[Mask when $u_0=0$.]{
		\includegraphics[height=3.5cm]{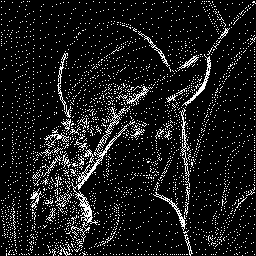}
	}
	\qquad
	\subfloat[Reconstruction when $u_0=0$. $\|f-u\|_{L^2(D)}=11.27$.]{
		\includegraphics[height=3.5cm]{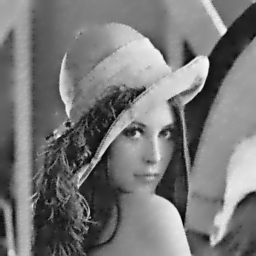}
	}
	\qquad
	\subfloat[Mask when $u_0=f$.]{
		\includegraphics[height=3.5cm]{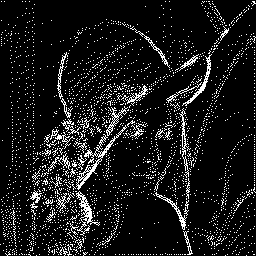}
	}
	\qquad
	\subfloat[Reconstruction when $u_0=f$. $\|f-u\|_{L^2(D)}=12.62$.]{
		\includegraphics[height=3.5cm]{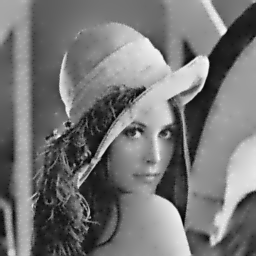}
	}
	
	\caption{Initial condition comparison for ``L2Insta'' method.}
	\label{fig:appendix:l2insta}
\end{figure}

For our three methods, the choice $u_0 = 0$ seems to be the best one. That is why we will use this initial condition for our experiments.